\documentclass[11pt]{article}

\usepackage{tgtermes}
\usepackage[T1]{fontenc}
\usepackage{amsmath,amsthm}
\usepackage{amsfonts,amssymb}
\usepackage{graphicx}
\usepackage{xcolor}

\usepackage{arxivs}
\usepackage[utf8]{inputenc}

\title{Relaxation of stochastic dominance constraints via optimal mass transport}

\author{Darinka Dentcheva \\
        Department of Mathematical sciences \\
        Stevens Institute of Technology \\
        Hoboken, NJ 07030, USA \\
        \texttt{darinka.dentcheva@stevens.edu}
        \And 
        Yunxuan Yi \\
        Department of Mathematical sciences \\
        Stevens Institute of Technology \\
        Hoboken, NJ 07030, USA \\
        \texttt{yyi1@stevens.edu}
        }

\date{}

\DeclareMathOperator{\di}{d}
\DeclareMathOperator{\dist}{dist}
\DeclareMathOperator{\conv}{co}
\DeclareMathOperator{\sign}{sign}
\renewcommand{\epsilon}{\varepsilon}

\newcommand{\Rb}{\mathbb{R}}

\newcommand{\Eb}{\mathbb{E}}
\newcommand{\Nb}{\mathbb{N}}

\newcommand{\Ac}{\mathcal{A}}
\newcommand{\D}{{\rm D}}

\newcommand{\Fc}{\mathcal{F}}
\newcommand{\Lc}{\mathcal{L}}
\newcommand{\Kc}{\mathcal{K}}

\newcommand{\Xc}{\mathcal{X}}

\newcommand{\Pc}{\mathcal{P}}

\newcommand{\Xf}{\mathfrak{X}}

\newcommand{\ssd}{\succeq_{(2)}}

\newcommand{\one}{\hspace{-0.1em}\mathbb{1}}

\newtheorem{definition}{Definition}
\newtheorem{theorem}{Theorem}
\newtheorem{lemma}{Lemma}
\newtheorem{corollary}{Corollary}

\newtheorem{assumption}{Assumption}

\begin{document}
\maketitle

\begin{abstract}
Optimization problems with stochastic dominance constraints provide a possibility to shape risk by selecting a benchmark random outcome with a desired distribution. The comparison of the relevant random outcomes to the respective benchmarks requires functional inequalities between the distribution functions or their transforms. A difficulty arises when no feasible decision results in a distribution that dominates the benchmark. Our paper addresses the problem of choosing a tight relaxation of the stochastic dominance constraint by selecting a feasible distribution, which is closest to those dominating the benchmark in terms of mass transportation distance. For the second-order stochastic dominance in a standard atomless space, we obtain new explicit formulae for the Monge-Kantorovich transportation distance of a given distribution to the set of dominating distributions.  We use our results to construct a numerical method for solving the relaxation problem.  Under an additional assumption, we also construct the associated projection of the distribution of interest onto the set of distributions dominating the benchmark.  For the stochastic dominance relations of order $r\in[1,\infty)$, we show a lower bound for the relevant mass transportation distance. Our numerical experience illustrates the efficiency of the proposed approach.  
\end{abstract}



\section{Introduction}

Stochastic orders define comparison between random variables.  The first stochastic ordering, which is widely used until today, is introduced in the seminal paper \cite{lorenz1905methods}. This notion is related to the theory of weak majorization, see, e.g., \cite{arnold2012majorization} and is known as the second order stochastic dominance. A comparison of random variables with respect to the first-order stochastic dominance is used in the context of statistical tests in the early works \cite{mann1947test,Blackwell1951ComparisonOE,lehmann1952testing}. 
The second-order stochastic dominance  has received considerable attention in economics and in insurance due to the fact that it is consistent with risk-averse preferences. It also  enjoys many generalizations and extensions to vector-valued outcomes and processes. We credit some early works drawing attention to stochastic dominance without claiming to be exhaustive  \cite{quirk1962admissibility,1965PCFD,hanoch1969efficiency,hadar1969rules,rothschild1978increasing,Fishburn1976ContinuaOS,Blackwell1951ComparisonOE,Stoyan1969MonotonieeigenschaftenDK,rolski1976comparison,Stoyan1986DietrichC,Whitt1986ComparisonMF,Szekli1986OnTC}.
A thorough survey and analysis of stochastic order relations is contained in the monographs \cite{MuellerStoyan} and \cite{shaked2007stochastic}.

Stochastic dominance plays a prominent role in decision making under uncertainty.
Optimization problems with stochastic dominance constraints were first introduced in \cite{ddarsiam} and further analyzed in \cite{ddarmp}. Optimization with stochastic dominance constraints relate to many other risk-averse models such as optimization using coherent measures of risk (\cite{dentcheva2008duality}), utility functions (\cite{ddarsiam}), distortions (\cite{Dentcheva2006InverseSD}), chance constraints, or Average (Conditional) Value-at-Risk constraints (\cite{ddarport}). 
 The basic optimization problem reads as follows.
\begin{equation}
\label{dc-pgeneral}
\min \, f(x)\quad 
\text{s.t. }\,  G(x)\succeq Y,\quad x\in\Xc.
\end{equation}
Here $\Xc$ is a closed convex subset of the decision space $\Xf$, which is assumed to be a separable Banach space.
The random variable $Y\in \Lc_r(\Omega,\Fc,P)$ with $r\geq 0$ is a selected benchmark outcome. 
 The notation $\Lc_r(\Omega,\Fc,P)$ refers to the space of random variables with finite $r$-th moments that are indistinguishable on events of $P$-measure zero. For $r\in (0,1)$, the set $\Lc_r(\Omega,{\mathcal F},P)$ contains those random variables $Z$, for which $\Eb[|Z|^r]$ is finite although $\big(\Eb[|Z|^r]\big)^\frac{1}{r}$ is not a norm. The set $\Lc_0(\Omega,{\mathcal F},P)$ consists of all $\Fc$-measurable mappings from $\Omega$ to $\Rb$, which are indistinguishable on sets of $P$-measure zero. The mapping $G$ assigns to each decision $x\in\Xc$ a random variable such that for each $x\in\Xc$, $G(x)\in\Lc_r(\Omega,\Fc,P)$ holds.

The objective function  $f:\Xf\to \Rb$ is continuous and the relation $\succeq$ is a suitably chosen stochastic dominance relation. For extensive analysis and numerical methods addressing problem \eqref{dc-pgeneral}, we refer the reader to \cite{DDARriskbook}. 

 Frequently, a decision maker can point to a benchmark random outcome $Y$ with an acceptable distribution, for which also data is readily available. In this case, the dominance relation in problem \eqref{dc-pgeneral} controls risk by requiring that the random outcome of interest is ``better'' than the benchmark in the sense of the relation $\succeq$. 
 While in many practical situations, natural benchmarks are available, other applications may lack such options. For example, it is difficult to create a suitable benchmark for two-stage stochastic problems with a dominance constraint on the recourse function. This type of problems are discussed in \cite{gollmer2011note,gollmer2008stochastic,Dentcheva2012TwostageSO}. Even more difficult is the task of choosing a suitable benchmark process for the multi-stage sequential decision problems analyzed in \cite{Escudero2016,Escudero2018time,dentcheva2022risk} or in the setting presented in \cite{haskell2013stochastic}. 
 In other situations, an ideal benchmark might be selected, that might not be achievable but it is desirable. In that case, the optimization problem becomes infeasible but it is of great interest to determine a way of approaching the set of distributions dominating the ideal one. An example of this situation are problems related to medical applications such as radiation therapy design, see e.g., \cite{vitt2023deepest}. Motivated by various applications when the stochastic dominance relation is not satisfied, many researchers have proposed relaxations and approximations of the stochastic dominance relation. We refer to \cite{leshno2002preferred,lizyayev2012tractable} and \cite{junova2025measures} and to the references ibid. 
 In this paper, we propose a mathematically sound and simultaneously numerically efficient method to identify the weakest relaxation of the stochastic dominance constraint if the benchmark turns out to be too demanding. We focus specifically on the second-order stochastic dominance and point out that most of the results translate in a straightforward manner to the increasing-convex order constraints.   We base the relaxation of the dominance constraint on the optimal mass transportation by including the transportation distance as a penalty term into the objective function. In that setting, the decision maker is seeking  another variable $Z\succeq Y$, whose distribution is closest to that of $G(x)$ in terms of an optimal mass transportation distance. As the calculation of those distances is often computationally demanding, we obtain two new formulae for the Monge-Kantorovich distance (also frequently called Wasserstein distance) of first order when second-order stochastic dominance is used in the optimization problem. The projection is explicitly characterized under an additional assumption, which is always satisfied for data-driven optimization problems and for two-stage problems. For more information on optimal mass transportation theory, we refer to \cite{rachev2006mass}. 
Using our results, we propose an efficient numerical method for solving the optimization problem with the relaxed constraint. 
The method is based on the new calculation of the mass transportation distance and its sequential approximations. If the optimization problem with the given benchmark is feasible, then the proposed method can be viewed as an exact penalty method  with the transportation distance constituting the exact penalty term. 

Our paper is organized as follows. In section \ref{s:sd}, we recall the notions and results associated with the stochastic dominance relation.  In section \ref{s:lower_bound}, we focus on the optimization problem of calculating the transportation distance of a given distribution to the set of distributions stochastically dominating the benchmark distribution. We obtain a new formulation of the problem and establish  a lower bound for the transportation distance to the set of distributions (corresponding to random variables) dominating the benchmark in a given order. Focusing on the second order stochastic dominance relation, we express the lower bound in two equivalent ways: one formula uses the shortfall functions and another one uses the Lorenz functions. In section \ref{s:projection}, making a simplifying assumption, we provide a construction for the random variable that plays the role of a projection onto the set of variables with distributions dominating the benchmark. We establish that the lower bound is, in fact, equal to the mass transportation distance for the second order stochastic dominance relation. In section \ref{s:generalization}, we generalize the formulae for the transportation distance to the cases that are ruled out by the simplifying assumption of section \ref{s:projection}. In section \ref{s:method}, we present a numerical method for solving the problem and analyze its convergence. Finally, in section \ref{s:numerical}, we report the results of our numerical experience, which illustrate the efficiency of the proposed relaxation approach. 

\section{Stochastic Dominance}
\label{s:sd}

Given a family $\mathcal{U}$ of functions $u: \mathbb{R} \to \mathbb{R}$, we say that a random variable $X$ dominates another random variable $Y$, if 
\[
  \mathbb{E}[u(X)] \geq \mathbb{E}[u(Y)]\quad \text{for all } u \in \mathcal{U}. 
  \]
The set of functions $\mathcal{U}$ is called the generator of the order and we denote the relation as $X \succeq_{\mathcal{U}} Y$. Commonly used orders are the stochastic dominance of first and second order. Stochastic dominance of first order is generated by all non-decreasing functions $u: \mathbb{R} \to \mathbb{R}$. This relation can also be characterized  by using the distribution functions of the random variables involved, as well as their inverses (see, e.g., \cite[Chapter 5]{DDARriskbook}). More precisely, for a random variable $Z$, let $F_Z(\cdot)$ be its right-continuous distributions function: $F_Z(\eta) = P(Z\leq \eta)$  for $\eta\in\Rb$. The quantile function $F^{(-1)}_Z$ of $Z$ is defined by setting for any $p\in (0,1)$, $F^{(-1)}_Z (p) = \inf\{\eta: F_Z(\eta)\geq p\}.$ 
\begin{definition}
A random variable $X$ dominates a random variable $Y$ with respect to the first order stochastic dominance,
denoted $X \succeq_{(1)} Y$, if and only if
 $F_X(\eta) \leq F_Y(\eta)$ for all $\eta \in \mathbb{R}$ or, equivalently, 
 $F^{(-1)}_X(p) \geq F^{(-1)}_Y(p)$ for all $p \in (0,1).$ 
 \end{definition}
Most popular relation is the second-order stochastic dominance due to its consistency with risk-averse preferences. This relation is  generated by all non-decreasing concave functions $u: \mathbb{R} \to \mathbb{R}$. For an integrable random variable $Z$, we define
\begin{equation}
\label{second}
 F^{(2)}_Z(\eta) = \int_{-\infty}^{\eta} F_Z(t) \ dt
 \quad \mbox{for} \  \eta \in \Rb.
\end{equation}
Changing the order of integration in \eqref{second}, we obtain the following representation
\begin{equation}
\label{shortfall}
F^{(2)}_Z(\eta)
=   \int_{-\infty}^\eta (\eta - t) \ P_Z(d t) =
 \Eb  [ \max ( \eta - Z, 0 )],
\end{equation}
with $P_Z = P\circ Z^{-1}$ being the induced probability measure on $\Rb$.
The integrated quantile function, $F^{(-2)}_{Z} : \Rb \rightarrow {\Rb }$,  of a random variable $Z \in \Lc_1(\Omega,\Fc,P)$ 
is given by
\begin{equation}
\label{F-2}
  F^{(-2)}_Z(p) =\int_0^p F^{(-1)}_Z(t)\; dt  \quad \mbox{for } \;  0 < p \le 1.
\end{equation}
We also define $F^{(-2)}_Z(0) = 0$ and $F^{(-2)}_Z(p) = +\infty$ for all  real numbers $p\not\in[0,1]$.
The function $F^{(-2)}_Z(\cdot)$  is known as the absolute Lorenz function, introduced in \cite{lorenz1905methods}.
\begin{definition}
A random variable $X\in \Lc_1(\Omega,\Fc,P)$ \emph{dominates in the second order} another
random variable $Y\in \Lc_1(\Omega,\Fc,P)$, denoted $X \succeq_{(2)} Y$, if one of the following equivalent requirements is satisfied  (cf. \cite[Chapter 5]{DDARriskbook}): 
\begin{itemize}
    \item[{\rm(F)}] $F^{(2)}_X(\eta) = \mathbb{E}[(\eta-X)_{+}]\leq \mathbb{E}[(\eta-Y)_{+}] = F^{(2)}_Y(\eta)$, for all $\eta \in \mathbb{R}$; 
    \item[{\rm(L)}] $F^{(-2)}_X(p) \geq F^{(-2)}_Y(p)$, for all $p \in [0,1]$.
\end{itemize}
\end{definition}
Motivated by the representation \eqref{shortfall}, we define the following function for a real number $r\geq 1$ and any random variable
$Z\in \Lc_{r-1}(\Omega,\Fc,P)$, :  
\begin{equation}
\label{e:Fp}
F^{(r)}_X(\eta)=\frac{1}{\Gamma(r)}\int_{-\infty}^{\eta}(\eta-t)^{r-1}\ P_{X}(d t)\\
= \frac{1}{\Gamma(r)}\Eb\big[\max(0,\eta-X)^{r-1}\big]
\end{equation}
with the convention that $\max(0,\eta-x)^0 = \one_{(-\infty,\eta]}(x)$, $F^{(1)}_Z(\cdot)= F_Z(\cdot),$ and with
\[
\Gamma(r) = \int_0^\infty z^{r-1} e^{-z}\  d z.
\]
\begin{definition}
\label{d:sdp}
A random variable $X\in \Lc_{r-1}(\Omega,\Fc,P)$ with $r\in [1,\infty)$ \emph{dominates in the $r$-th order} another
random variable $Y\in \Lc_{r-1}(\Omega,\Fc,P)$ if 
\begin{equation}
\label{d:psd}
F^{(r)}_X(\eta) \le F^{(r)}_Y(\eta)\quad \text{for all } \eta\in \Rb.
\end{equation}
\end{definition}
This unifying notion allowing also fractional orders $r$ was introduced in \cite{Fishburn1976ContinuaOS}.
The functions $F^{(r)}_Z(\cdot)$ for $r\geq 2$ are evidently convex as we shall argue in due course. 

The stochastic dominance of order $r\geq 2$ reflects preference to large values. If the decision maker has a preference to small values, than a popular stochastic order is the increasing convex order, denoted $\preceq_{\rm ico}$.   For two random variables $X,Y \in \Lc_1(\Omega,\Fc,P)$, it is said that $X$ is stochastically smaller than $Y$ with respect to the increasing convex order if 
\[
  \bar{F}_X^{(2)}(\eta) = \mathbb{E}[(X-\eta)_{+}]\leq \mathbb{E}[(Y-\eta)_{+}]=\bar{F}_Y^{(2)}(\eta)\quad  \text{for all } \eta \in \mathbb{R}.
\]
This order is closely related to the stochastic dominance of second order in the following way:
\begin{equation}
\label{e:ico}
  X\preceq_{\rm ico} Y \quad\Longleftrightarrow\quad -X\succeq_{(2)} -Y. 
\end{equation} 
For more general definition and further information about stochastic comparisons when small values are preferred, see \cite[Chapter 5]{DDARriskbook}.

Let a benchmark random variable $Y\in\Lc_{r-1}(\Omega,\Fc,P)$, $r\geq 1,$ be given. We consider the sets
\[
A_r (Y) =\big\{ X\in \Lc_{r-1}(\Omega,\Fc,P):\, X\succeq_{(r)}  Y\big\}.
\]
For any real numbers $q>r\geq 1$, the following inclusion holds (\cite[Theorem 5.3]{DDARriskbook}):
\begin{equation}
\label{e:Ar-inclusions}
A_{r}(Y) \cap \Lc_{q-1}(\Omega,\Fc,P)  \subset A_q(Y).
\end{equation}
Our goal is to propose a mathematically sound and numerically efficient way of solving the relaxation of problem \eqref{dc-pgeneral} with the main focus on the second-order stochastic dominance.  We use  
a suitable mass transportation distance $\di(\cdot,\cdot)$ between two random variables from $\Lc_{r-1}(\Omega,\Fc,P)$ and a weight-parameter $\alpha>0$ for the distance of $G(x)$ to the dominating set $A_r(Y)$. 

We focus on the analysis and the numerical solution of the relaxation problem:  
\begin{equation}
 \label{p:relaxedproblem}
\min_{x,Z} \,  f(x) + \alpha \di(G(x),Z) \quad 
\text{s.t. }\, Z \succeq_{(r)} Y, \quad x \in \mathcal{X}.
\end{equation}
The objective function might have the form $f(x) = -\Eb[G(x)]$ or is might represent a coherent measure of risk for $-G(x)$.  
If $r=2$, then the representation (L) in the definition of the relation $Z \succeq_{(2)} Y$ is equivalent to a continuum of Average Value-at Risk inequalities for all levels of risk. Imposing a stochastic order constraint is a known way of shaping the risk profile of the random outcome $G(x)$. On the other hand, if we only maximize the expected value of $G(x)$, we obtain a larger value at the expense of higher risk. Hence, problem \eqref{p:relaxedproblem} provides a way to balance risk and reward also in the case when the benchmark is achievable.  

\section{Lower bound of the transportation distance}
\label{s:lower_bound}

The main problem we analyze in this section is the problem of calculating the transportation distance for a given decision $x\in\Xc$.  
We shall use the Monge-Kantorovich transportation distance $W_r(\cdot,\cdot)$ (also called Wasserstein distance). 
Let the space $\Pc_r(\Rb)$ contain the probability measures on the real line which have a finite moment of order $r>0$ while $\Pc_0(\Rb)$ stands for the space  of probability measures on the real line equipped with the weak convergence. 
We equip the space $\Pc_r(\Rb)$ with the distance $W_\ell(\cdot,\cdot)$, where $1\leq\ell\leq r.$ For scalar random variables, $W_\ell(\cdot,\cdot)$ has the following representation (cf. \cite{dallaglio}):
\[
W_\ell(X,Z) = \Big(\int\limits_0^1 \big| F^{-1}_Z(p) - F^{-1}_X(p)\big|^\ell\; dp\Big)^\frac{1}{\ell}. 
\]
We define $\dist_\ell\big(X,A_{r}(Y)\Big)$ with $1\leq\ell\leq r+1$ by setting
\begin{equation}
\label{d:distancetoA}
\dist_\ell\big(X,A_r(Y)\big) = \inf\big\{ W_\ell(X,Z):\; Z\in A_r(Y)\big\}.
\end{equation}
For $\ell=1$, we shall omit the subscript and shall use simply $W(\cdot,\cdot)$ and $\dist(\cdot,\cdot)$. 

While the mass transportation distances are defined on the space of probability measures, in our optimization problems we deal with random outcomes. We identify those random variables with elements $Z\in \Lc_r(\Omega,\Fc,P)$ and assume throughout the paper that the probability space $(\Omega,\Fc,P)$ is atomless. In that case, a random variable $U$ exists, which is defined on $(\Omega,\Fc,P)$ and has the uniform distribution on $[0,1]$.  
 Every probability distribution $\mu\in \Pc_r(\Rb)$ is determined by the distribution function: $F_{\mu}(\eta)  = \mu\big((-\infty,\eta]\big)$, which  is
non-decreasing and right-continuous. Therefore, the inverse of $F_{\mu}(\cdot)$ may be defined as follows:
\begin{equation}
\label{quantile-function}
F^{(-1)}_\mu(p) = \inf\,\{ \eta\in \Rb :\; F_{\mu}(\eta) \ge p\},\quad p\in (0,1).
\end{equation}
The definition implies that $F^{-1}_\mu(p)$ is the smallest $p$-quantile of $\mu$. We call $F^{-1}_\mu(\cdot)$ the quantile function associated with the probability measure $\mu$.
Every quantile function is nondecreasing and left-continuous on the open interval $(0,1)$. On the other hand, every nondecreasing and left-continuous
function $Q(\cdot)$ on $(0,1)$ uniquely defines a distribution function:
$F_\mu(\eta) = \sup\, \{ p\in (0 ,1) : Q(p) \le \eta \},
$
which corresponds to a certain probability measure $\mu \in \Pc(\Rb)$.
Furthermore, for any random variable $Z\in \Lc_{r}(\Omega,\Fc,P)$, the variable 
$\tilde{Z} = F_Z^{(-1)}(U)$ has the same distribution as $Z$.  

Consider the sets $\Ac_{r+1}(P_Y)$ of all probability measures $\mu\in \Pc_r(\Rb)$ dominating the measure $P_Y$ induced by the random variable $Y\in \Lc_r(\Omega,\Fc,P)$. More precisely, for $i$ such that $1\leq i\leq r+1$, we define
\[
\Ac_i (P_Y) =\big\{ \mu\in \Pc_{i-1}(\Rb):\, F^{(i)}_\mu(\eta) \leq  F^{(i)}_Y(\eta)\big\}.
\]
\begin{theorem}
\label{t:equivalence}
The mapping $P_Y\to \Ac_r(P_Y)\subset \Pc_{r-1}(\Rb) $ has a closed graph for any $r\geq 1$ and it has convex images for $r\geq 2.$ 
\end{theorem}
\begin{proof}
To show the convexity of the sets $\Ac_r(P_Y)$, we consider a convex combination of any two measures 
$\mu$ and $\nu$ in $\Ac_r(P_Y)$: $\lambda= \alpha\mu+ (1-\alpha)\nu$. Using an uniform random variable $U$ defined on 
$(\Omega,\Fc,P),$ with values in $[0,1]$, we define random variables $Z_\lambda = F^{(-1)}_\lambda(U)$, $Z_\mu = F^{(-1)}_\mu(U)$, and $Z_\nu = F^{(-1)}_\nu(U)$. The distribution functions of $Z_\lambda $, $Z_\mu$, and $Z_\nu$ are 
$\lambda,\mu$ and $\nu$, respectively. Hence, $Z_\mu$, and $Z_\nu$ are elements of $A_r(Y)$, which is a convex set for $r\geq 2$, entailing that $Z_\lambda\in A_r(Y)$. This implies, that $\lambda\in\Ac_r(P_Y).$

Now, let $(P_{Y_n},\mu_n)$, $n\in\Nb$  be such that $\mu_n\in \Ac_r(P_{Y_n})$ and the sequences $P_{Y_n}$ and $\mu_n$ are Cauchy sequence of measures in $\Pc_{r-1}(\Rb)$ converging to $\nu$ and $\bar{\mu}$, respectively. In the same way, we define random variables
$Z_n = F^{(-1)}_{\mu_n}(U)$ and obtain that $Z_n\in A_r(Y_n).$ Let $\bar{Z} = F^{(-1)}_{\bar{\mu}}(U)$ and $\bar{Y} = F^{(-1)}_{\nu}(U)$. For $r>1$, we have 
\[
W_{r-1}(Z_n,\bar{Z}) = \Big(\int_0^1 \big| F^{(-1)}_{\mu_n}(p) - F^{(-1)}_{\bar{\mu}}(p)\big|^{r-1}\; dp \Big)^\frac{1}{r-1}\xrightarrow{n\to \infty} 0. 
\]
This implies that 
\[
\|Z_n-\bar{Z}\|_{r-1}= \Big(\int_\Omega \big| F^{(-1)}_{\mu_n}(U(\omega)) - F^{(-1)}_{\bar{Z}}(U(\omega))\big|^{r-1}\; P(d\omega) \Big)^\frac{1}{r-1} \xrightarrow{n\to \infty} 0.
\]
Hence, $Z_n$ converges to $\bar{Z}$ in $\Lc_{r-1}$-norm. Similarly, $Y_n$ converges to $\bar{Y}$ in $\Lc_{r-1}$-norm. Then Lemma 5.27 in \cite{DDARriskbook} implies that $\bar{Z}\in A_r(\bar{Y})$, which entails that $\bar{\mu}\in \Ac_r(P_Y).$ More precisely, this follows from the fact that for any $\eta\in \Rb$, the operator $Z \mapsto F^{(r)}_Z(\eta)$ is continuous in $\Lc_{r-1}(\Omega,\Fc,P)$, $r>1$ due to the continuity of the mappings $z\to (\eta -z)_+^{r-1}$ and $Z\to \Eb [Z]$.

For $r=1$, we apply  Lemma 5.26 in \cite{DDARriskbook} to arrive at the desired conclusion. 
\end{proof}

For any fixed $Y\in \Lc_r(\Omega,\Fc,P)$, the quantity $\dist_\ell(X,A_{r+1}(Y))$ with $1\leq\ell\leq r$ is well defined for any $X\in \Lc_r(\Omega,\Fc,P)$ since the set $A_{r+1}(Y)$ is non-empty and 
\[
0\leq \inf\big\{ W_\ell(X,Z):\; Z\in A_{r+1}(Y)\big\}\leq W_\ell(X,Y).
\] 
Theorem~\ref{t:equivalence} also shows that we can use the distance between distributions and between random variables interchangeably, which is important in the context of stochastic optimization problems. 

Now, we formulate the problem of identifying the weakest relaxation of the stochastic order constraints in problem \eqref{dc-pgeneral} when stochastic dominance of order $r$ is used.
We denote the random outcome of our decision by $X=G(x)\in\Lc_r(\Omega,\Fc,P)$ and consider $1\leq\ell\leq r$. We deal with the following projection problem:
\begin{equation}
\label{p:projection}
\inf_Z \;  W_\ell(X,Z) \quad
\text{ s.t. }\;  Z\in A_{r+1}(Y).
\end{equation}
We define the set $A_1^{(1)}(Y)= A_1(Y)\cap \Lc_1(\Omega,\Fc,P)$. The mapping $P_Y\to \Ac_r(P_Y)\subset \Pc_1(\Rb) $ has a closed graph and the set $A_1^{(1)}(Y)$ is closed with respect to the $\Lc_1$-norm. 
More generally, for $Y\in \Lc_r(\Omega,\Fc,P)$ and any numbers $1\leq i <\ell \leq r$, we introduce the restrictions
\[
A_i^{(\ell)}(Y)= A_i(Y)\cap \Lc_\ell(\Omega,\Fc,P)\quad\text{and}\quad \Ac_i^{(\ell)} (P_Y) =\Ac_i(P_Y)\cap \Pc_{\ell}(\Rb).
\]
The set $A_i^{(\ell)}(Y)$ is closed with respect to the $\Lc_\ell$-norm and the set $\Ac_i^{(\ell)}(P_Y)$ is closed with respect to the distance
$W_\ell(\cdot,\cdot)$ due to the properties of the $W_\ell$ distance. 
\begin{lemma}
\label{l:smallerdist}
Given random variables $X,Y\in \Lc_1(\Omega,\Fc,P)$, for every random variable $Z\in A_r^{(1)}(Y)$ with $r\in [1,2)$, an associated variable $\bar{Z}\in A_r^{(1)}(Y)$ exists such that 
$F^{(-1)}_{\bar{Z}}(p) \geq  F^{(-1)}_{X}(p)$ for all $p \in [0,1]$ and $W_1(X,\bar{Z})\leq W_1(X,Z)$. 

If $X,Y\in \Lc_{r-1}(\Omega,\Fc,P)$ with  $r\geq 2$, then  
for every random variable $Z\in A_r(Y)$ an associated variable $\bar{Z}\in A_r(Y)$ exists such that 
$F^{(-1)}_{\bar{Z}}(p) \geq  F^{(-1)}_{X}(p)$ for all $p \in [0,1]$ and $W_\ell(X,\bar{Z})\leq W_\ell(X,Z)$ for any $1\leq \ell\leq r-1$.
\end{lemma}
\begin{proof}
Consider an arbitrary $Z\in A_r(Y)$. If $F^{(-1)}_{Z}(p) \geq  F^{(-1)}_{X}(p)$, then there is nothing to prove because we can set $\bar{Z}=Z$.  Assume that $p\in(0,1)$ exists such that $F^{(-1)}_{Z}(p) <  F^{(-1)}_{X}(p)$. 
Due to the monotonicity and left-continuity of the quantile function, an interval $(\alpha,\beta)\subset (0,1)$ of positive length exists such that $F^{(-1)}_{X}(p) -  F^{(-1)}_{Z}(p)> 0$ for all $p\in(\alpha,\beta)$.
Then 
\[
\int^1_0 \max\big(0,  F^{(-1)}_{X}(p) - F^{(-1)}_{Z}(p) \big) \, d p > 0.
\]
We define $Q(p) = \max\big(F^{(-1)}_{Z}(p), F^{(-1)}_{X}(p)\big)$ and notice that it is a non-decreasing and left-continuous function. Let the random variable $\bar{Z}$ be defined as $\bar{Z}= Q(U)$, where $U$ is a random variable with the uniform distribution on $[0,1]$. We obtain
\begin{align*}
W_1(\bar{Z}, X) & = \int^1_0 |Q(p) - F^{(-1)}_{X}(p)| \, d p = \int^1_0 \big(Q(p) - F^{(-1)}_{X}(p)\big) \, d p \\
& = \int^1_0 \max\big(F^{(-1)}_{Z}(p), F^{(-1)}_{X}(p)\big) - F^{(-1)}_{X}(p) \, d p \\
& = \int^1_0 \max\big(F^{(-1)}_{Z}(p) - F^{(-1)}_{X}(p), 0\big) \, d p \\
& < \int^1_0 \big|F^{(-1)}_{Z}(p) - F^{(-1)}_{X}(p)\big| \, d p = W_1(Z,X).
\end{align*}
In the same way, it follows that for $r-1\geq\ell>1$ the following relations hold:
\begin{multline*}
[W_\ell(\bar{Z}, X)]^\ell  = \int^1_0 |Q(p) - F^{(-1)}_{X}(p)| ^\ell\, d p \\ < \int^1_0 \Big(\max\big\{F^{(-1)}_{Z}(p) - F^{(-1)}_{X}(p),  F^{(-1)}_{X}(p) - F^{(-1)}_{Z}(p)\big\}\Big)^\ell \, d p \\
 = \int^1_0 |F^{(-1)}_{Z}(p) - F^{(-1)}_{X}(p)|^\ell \, d p = [W_\ell (Z,X) ]^\ell.
\end{multline*}
Additionally, $\bar{Z}\succeq_{(1)} Z$ and $\bar{Z}\in\Lc_{r-1}(\Omega,\Fc,P)$ by construction. We conclude that $\bar{Z}\succeq_{(r)} Z\succeq_{(r)} Y$ for any $r \geq  1$ due to \eqref{e:Ar-inclusions}. 
Thus, $\bar{Z} \in A_r(Y)$ as stated. 
\end{proof}

\begin{corollary}
\label{c:equivalent_problem}
Problem \eqref{p:projection} is equivalent to the following problem:
\begin{equation}
\label{p:projection2}	
\min_{Z} \,  W_\ell(X,Z) \quad 
\text{s.t. }  Z \succeq_{(r)} Y,\quad  Z \succeq_{(1)} X.
\end{equation}
with $1\leq \ell\leq r-1$.
Additionally, for $\ell=1$, the objective function of problem \eqref{p:projection2} takes on the form
\[
W_1(Z,X)=\mathbb{E}[Z] - \mathbb{E}[X].  
\]
\end{corollary}
\begin{proof}
The expression for the objective function follows directly from the definition of $W_1$ taking into account that $F^{(-1)}_{Z}(p) \geq F^{(-1)}_{X}(p)$.

The feasible set of problem \eqref{p:projection2} is included in the feasible set of \eqref{p:projection}. Hence, its optimal value is not smaller than the optimal value of \eqref{p:projection}.
This together with Lemma~\eqref{l:smallerdist} proves the statement.
\end{proof}

\begin{theorem}
\label{t:lower_bound}
Given random variables $X,Y\in\Lc_1(\Omega,\Fc,P)$, the following holds 
\[
 \dist(X,A_2(Y)) \geq \sup_{p\in [0,1]} \big(F^{(-2)}_Y(t) - F^{(-2)}_X(t)\big) = \sup_{\eta \in \Rb} \big(F^{(2)}_X(\eta) - F^{(2)}_Y(\eta)\big).
\]
\end{theorem}
\begin{proof}
Let $\epsilon >0$. The definition of $\dist(X, A_2(Y))$ implies that a random variable 
$Z_\epsilon\in A_2(Y)$ exists such that 
\[
\dist(X, A_2(Y)) +\epsilon \geq W (X,Z_\epsilon)
\]
Using Lemma~\ref{l:smallerdist}, we obtain
that a point $\bar{Z}_\epsilon\in A_2(Y)\cap A_1^{(1)}(X)$ exists such that for all $p\in (0,1]$:
\begin{align*}
W(X, Z_\epsilon)  & \geq W_1(X,\bar{Z}_\epsilon)= \int^1_0 \big(F^{(-1)}_{\bar{Z}_\epsilon}(t) - F^{(-1)}_X(t)\big)\; d t \\
& \geq \int^p_0 (F^{(-1)}_{\bar{Z}_\epsilon}(t) - F^{(-1)}_X(t))\; d t 
\; \geq \int^p_0 F^{(-1)}_{Y}(t)\; d t - \int^p_0 F^{(-1)}_X(t))\; d t. 
\end{align*}
The last inequality holds due to $\bar{Z}_\epsilon\ssd Y$. Hence, 
\[
W(X, Z_\epsilon) \geq \sup_{p\in (0,1]} \Big( \int^p_0 F^{(-1)}_{Y}(t)\; d t - \int^p_0 F^{(-1)}_X(t)\; d t\Big) = 
\sup_{p\in (0,1]} \big(F^{(-2)}_Y(t) - F^{(-2)}_X(t)\big). 
\]
Therefore, for any $\epsilon >0$, we obtain  
$\displaystyle{\dist(X, A_2(Y)) +\epsilon \geq \sup_{p\in (0,1]} \big(F^{(-2)}_Y(t) - F^{(-2)}_X(t)\big)}. $
Letting $\epsilon\downarrow 0$, we obtain the first lower bound.

Now, we shall show the equivalent formula for the lower bound of the distance.
For this purpose, we use the conjugate duality relation between the functions $F^{(2)}_Z(\cdot)$ and $F^{(-2)}_Z(\cdot)$, which holds for any integrable random variable $Z$ (\cite[Theorem 2.15]{DDARriskbook}).
For any $p\in(0,1]$, we denote a $p-$quantile of $Y$ by $\eta_Y(p)$. Due to the conjugate duality, the following relations hold
\begin{align*}
F^{(-2)}_Y(p) - F^{(-2)}_X(p) &= p\eta_Y(p) - F^{(2)}_Y(\eta_Y(p)) - 
\sup_{\eta\in\Rb} \big( p\eta- F^{(2)}_X(\eta)\big) \\
& \leq p\eta_Y(p) - F^{(2)}_Y(\eta_Y(p)) - p\eta_Y(p) + F^{(2)}_X(\eta_Y(p))\\
& \leq \sup_{\eta \in \Rb} \big(F^{(2)}_X(\eta) - F^{(2)}_Y(\eta)\big)
\end{align*}
Taking supremum with respect to $p\in [0,1]$, we obtain
\begin{equation}
\label{e:ineq1}
\sup_{p \in [0,1])} \big(F^{(-2)}_Y(p) - F^{(-2)}_X(p)\big)\leq \sup_{\eta \in \Rb} \big(F^{(2)}_X(\eta) - F^{(2)}_Y(\eta)\big).
\end{equation}
Now for any $\eta\in\Rb$ we define $p_Y(\eta)= P(Y\leq \eta)$. 
Using conjugate duality again, we obtain the following relations
\begin{align*}
F^{(2)}_X(\eta) - F^{(2)}_Y(\eta) & = \eta p_Y(\eta) - F^{(-2)}_Y(p_Y(\eta)) - 
\sup_{p\in[0,1])} \big(p\eta - F^{(-2)}_X(p)\big) \\
&\leq \eta p_Y(\eta) -  F^{(-2)}_Y(p_Y(\eta)) - 
\eta p_Y(\eta)+ F^{(-2)}_X(p_Y(\eta)) \\
& \leq \sup_{p\in[0,1])} \big(F^{(-2)}_Y(p) - F^{(-2)}_X(p)\big)
\end{align*}
 Taking supremum with respect to $\eta\in\Rb$ , we obtain
\begin{equation}
\label{e:ineq2}
\sup_{p \in [0,1])} \big(F^{(-2)}_Y(p) - F^{(-2)}_X(p)\big)\geq \sup_{\eta \in \Rb} \big(F^{(2)}_X(\eta) - F^{(2)}_Y(\eta)\big).
\end{equation} 
Relations \eqref{e:ineq1} and \eqref{e:ineq2} imply the result.
\end{proof}

\section{Projection}
\label{s:projection}

In this section, we focus on the set $A_2(Y)\subset\Lc_1(\Omega,\Fc,P)$ and consider a fixed random variable $X\in\Lc_1(\Omega,\Fc,P)$ such that $X\not\in A_2(Y)$. 
We construct the projection of the random variable $X$ onto $A_2(Y)$ under an additional assumption and we show that the distance $\dist(X,A_2(Y))$ is equal to the lower bound obtained in Theorem~\ref{t:lower_bound}.

We define the following two functions: 
\begin{align*}
\text{for all } p\in (0,1),\quad h(p) & = F^{(-1)}_Y(p) - F^{(-1)}_X(p),\\
\text{for all } p\in [0,1],\quad l(p) & = F^{(-2)}_Y(p) - F^{(-2)}_X(p).
\end{align*}
Additionally, we shall use the set $A$ and its complement $A^C$, defined as follows:
\[
A = \{ p\in(0,1] : \;\forall q < p, \; l(q) < l(p) \},\quad A^C = [0,1] \setminus A. 
\]
Let $\text{sign}(x)$ be sign of $x$ that ranges in $\{-1, 0, 1\}$.
Throughout this section we adopt an assumption, which we shall relax in the next section. 
\begin{assumption}
\label{a:finite}
 The function $p\mapsto\sign(h(p))$ has at most finitely many discontinuity points. 
\end{assumption}

\begin{lemma}
\label{l:hpositiveonA}
For any $p \in A$, the inequality $h(p) \geq 0$ holds. Equivalently, $h(p) < 0$ implies $p \in A^C$. 
\end{lemma}
\begin{proof}
Fix a number $p\in A.$
The function $h(\cdot)$ is left-continuous on $(0,1)$ by definition. Thus, a small interval $(p-\delta,p)$ exists, such that $h(\cdot)$ is continuous on this interval. Consequently, $l(\cdot)$ is differentiable on this interval and its derivative satisfies $l'(p) = h(p)$. We obtain
\begin{align*}
h(p)  & = \lim_{\epsilon \downarrow 0} h(p-\epsilon)  = \lim_{\epsilon \downarrow 0} l'(p-\epsilon) 
   = \lim_{\epsilon \downarrow 0} \frac{l(p) - l(p-\epsilon)}{\epsilon} \geq 0,
\end{align*}
as claimed. 
\end{proof}
\begin{theorem}
\label{t:A_finite_union}
Under assumption \ref{a:finite}, the sets $A$ and $A^C$ both consist of finitely many intervals.
\end{theorem}
\begin{proof}
We shall prove this claim by contradiction. The function $l(\cdot)$ is continuous and has directional derivatives as a difference of convex functions. 
Observe that if the statement is not true than both $A$ and $A^C$ consist of countably many intervals with some of the intervals being possibly a single point.
This means that an increasing sequence $p_1 < q_1 < p_2 < q_2 < \cdots$ exists, such that $p_i\in A$ and $q_i\in A^C.$ Due to the definition of the sets $A$ and $A^C$, we infer that a sequence
$\bar{p}_1 < \bar{q}_1 < \bar{p}_2 < \bar{q}_2 < \cdots$ exists, such that 
$p_1 < \bar{p}_1 \leq q_1 <\bar{q}_1 <  p_2 < \bar{p}_2 \leq  q_2 < \bar{q}_2 < \cdots$ and $l(\cdot)$ has a local maximum at $\bar{p}_i$ and a local minimum at $\bar{q}_i$. This entails that $h(\cdot)$ changes signed infinitely many times, which contradicts Assumption~\ref{a:finite}. 
\end{proof}

\begin{lemma}
\label{l:Q-non-decreasing}
The functions $Q(\cdot)$ and $\hat{Q}(\cdot)$ defined as follows
\begin{equation}
\label{e:projection1}
Q(p) = 
  \begin{cases}
    F^{(-1)}_Y(p), & p \in A.     \\
    F^{(-1)}_X(p), & p \in A^C;
  \end{cases}
  \quad \hat{Q}(p) = \lim_{\epsilon \downarrow 0} Q(p-\epsilon).
\end{equation}
are non-decreasing and $\hat{Q}(\cdot)$ is left-continuous. 
\end{lemma}
\begin{proof}
Consider two points $p_1,p_2\in (0,1)$ such that 
$ p_1<p_2$.  The claim is obvious when $p_1, p_2$ are both in $A$ or both in $A^C$. 

When $p_1 \in A^C, p_2 \in A$, using the definition of the set $A$, we obtain
 \[
 Q(p_1) = Q^{(-1)}_X(p_1) \leq Q^{(-1)}_X(p_2) \leq Q^{(-1)}_Y(p_2) = Q(p_2). 
 \]
 The first inequality follows by the monotonicity of the quantile function while the second inequality follows by Lemma~\ref{l:hpositiveonA}. 
 Hence, monotonicity follows in this case.

If $p_1 \in A, p_2 \in A^C$, then  a number $0<q<p_2$ exist, such that $l(q) \geq l(p_2)$. 
Then we observe that $\tilde{q}$ exists such that $0< q < \tilde{q} < p_2$ and $h(\tilde{q}) \leq 0$. 
Indeed, if this is not true, then we infer that $h(\tilde{q})>0$ on $(q,p_2)$, which  contradicts  $l(q) \geq l(p_2)$. If $p_1 \leq \tilde{q}$, then 
\[
Q(p_1) = Q^{(-1)}_Y(p_1) \leq Q^{(-1)}_Y(\tilde{q}) \leq Q^{(-1)}_X(\tilde{q}) \leq Q^{(-1)}_X(p_2) = Q(p_2),
\]
obtaining monotonicity again.
If $p_1 > \tilde{q}$, then $l(p_1) > l(q)\geq l(p_2)$. Thus $\tilde{q}' \in [p_1,p_2]$ exists, such that $h(\tilde{q}') \leq 0$. Then we proceed as in the case of $p_1\leq \tilde{q}'$. 
This completes the proof for $Q(\cdot)$.

Existence of the limit in the definition of $\hat{Q}$ follows from Lemma~\ref{l:hpositiveonA} and monotonicity is preserved when taking the limit. 
The left-continuity of $\hat{Q}$ is guaranteed by construction.
\end{proof}

\begin{theorem}
\label{t:hatZssdY}
Let the random variable $\hat{Z}$ be defined by setting $\hat{Z}= \hat{Q}(U)$, where $U$ is the uniform random variable. 
Assuming \ref{a:finite}, the random variable $\hat{Z}$ dominates $Y$ in the second order, i.e., 
$F^{(-2)}_{\hat{Z}}(p) \geq F^{(-2)}_Y(p)$ for all  $p \in [0,1]$. 
\end{theorem}
\begin{proof}
Since $A$ and $A^C$ split $[0,1]$ in finitely many intervals, we shall argue by going over those intervals consecutively.
Consider the first interval after 0. If it is in $A$, then it follows that $F^{(-1)}_{\hat{Z}}(p) = F^{(-1)}_{Y}(p)$ and $F^{(-2)}_{\hat{Z}}(p) = F^{(-2)}_{Y}(p)$ on this interval. 

If it is in $A^C$, then $F^{(-1)}_{\hat{Z}}(p) = F^{(-1)}_{X}(p)$, $F^{(-2)}_{\hat{Z}}(p) = F^{(-2)}_{X}(p)$, with $l(p) = F^{(-2)}_{Y}(p) - F^{(-2)}_{X}(p) \leq l(0) = 0$. 
Hence, $F^{(-2)}_{\hat{Z}}(p) \geq F^{(-2)}_{Y}(p)$ on this interval. 

Now take $0 \leq p_1 < p_2 \leq 1$, when $p_{1}$ and $p_2$ fall into adjacent intervals (e.g. $p_1$ is in the first interval and $p_2$ in the second interval.) 
Suppose $F^{(-2)}_{\hat{Z}}(p_1) \geq F^{(-2)}_Y(p_1)$. \\
When $p_1 \in A$ and  $p_2 \in A^C$, then a point $q$ exists, such that $[p_1, q] \subset A$, $(q, p_2] \subset A^C$. We infer that 
\[
l(p_2) = F^{(-2)}_Y(p_2) - F^{(-2)}_X(p_2) \leq F^{(-2)}_Y(q) - F^{(-2)}_X(q) = l(q). 
\]
The last inequality implies that 
\[
F^{(-2)}_X(p_2) - F^{(-2)}_X(q) \geq F^{(-2)}_Y(p_2) - F^{(-2)}_Y(q). 
\]
We obtain the following chain of relations
\begin{align*}
F^{(-2)}_{\hat{Z}}(p_2) & = F^{(-2)}_{\hat{Z}}(q)  + \int^{p_2}_{q} \hat{Q}(p) d p  
        =  F^{(-2)}_X(p_2) - F^{(-2)}_X(q) + F^{(-2)}_{\hat{Z}}(q)\\
       & \geq F^{(-2)}_Y(p_2) - F^{(-2)}_Y(q) + F^{(-2)}_{\hat{Z}}(q) 
        \geq F^{(-2)}_Y(p_2) - F^{(-2)}_Y(p_1) + F^{(-2)}_{\hat{Z}}(p_1) \\
        &\geq F^{(-2)}_Y(p_2).
\end{align*}
We infer that $F^{(-2)}_{\hat{Z}}(p_2)\geq F^{(-2)}_Y(p_2).$

In the case of $p_1 \in A^C$, $p_2 \in A$, we can identify the number $q$ such that $[p_1, q] \subset A^C$, $(q, p_2] \subset A$. Then we get $F^{(-2)}_{\hat{Z}}(q) \geq F^{(-2)}_Y(q)$ either because $q$ is in the first interval, or by the conclusions from the previous interval. Using the definition of $\hat{Z}$ and its quantile function, we obtain  
\begin{align*}
F^{(-2)}_{\hat{Z}}(p_2) & = \int^{p_2}_{q} F^{(-1)}_{\hat{Z}}(p) d p + F^{(-2)}_{\hat{Z}}(q) 
			  =  F^{(-2)}_Y(p_2) - F^{(-2)}_Y(q) + F^{(-2)}_{\hat{Z}}(q)\\
			 & \geq F^{(-2)}_Y(p_2) - F^{(-2)}_Y(q) + F^{(-2)}_{Y}(q) = F^{(-2)}_Y(p_2).
\end{align*}
Due to the finite number of interval, we conclude that  $F^{(-2)}_{\hat{Z}}(p) \geq F^{(-2)}_Y(p)$ holds for all $p\in (0,1]$.
\end{proof}

\begin{theorem}
\label{t:hatZisprojection}
The random variable $\hat{Z}$ defined in Theorem~\ref{t:hatZssdY} is the projection of $X$ onto $A_2(Y)$ under Assumption \ref{a:finite}. Furthermore,
\begin{equation}
\label{e:dist-final}
\begin{aligned}
\dist(X,A_2(Y)) & = W(X,\hat{Z}) = \sup_{p \in [0,1]} \big(F^{(-2)}_Y(p) - F^{(-2)}_X(p)\big)\\ 
& =
\sup_{\eta \in \Rb} \big(F^{(2)}_X(\eta) - F^{(2)}_Y(\eta)\big).  
\end{aligned}
\end{equation}
\end{theorem}
\begin{proof}
Consider the following function: 
\[ 
G(p) = \sup_{q \in [0,p]} l(q) = \sup_{q \in [0,p]}  \int^q_0 h(t) d t = \sup_{q \in [0,p]} l(q).
\]
Observe that $G(p) = l(p)$ for all $p \in A$, while for $p \in A^C$, the mapping $G(\cdot)$ as a constant on each interval of $A^C$. 
Since $l(\cdot)$ is differentiable, we infer that $G(\cdot)$ is differentiable on $[0,1]$ except for finitely many points, with derivative 
\begin{equation}
G'(p) = 
\begin{cases}
		h(p), & \text{ for } p \in A     \\
		0, & \text{ for } p \in A^C
	\end{cases}.
\end{equation} 
Notice that the definition of $\hat{Z}$ implies the following
\begin{align*}
W(X,\hat{Z}) & = \int_{[0,1]} |F^{(-1)}_{\hat{Z}}(p) - F^{(-1)}_{X}(p)|\; d p 
= \int_{A \cap [0,1]} F^{(-1)}_{Y}(p) - F^{(-1)}_{X}(p)\; d p \\ 
& =  \int_{A \cap [0,1]} h(p)\; d p. 
\end{align*}
The mapping $w: p\to \int_{A \cap [0,p]} h(t)\; d t $ is also differentiable a.e. on $[0,1]$. Denoting the indicator function of the set $A$ by $\one_A(\cdot)$, we express the derivative  $w'(p)=h(p) \one_A(p)$. It is equal to $G'(p)$ a.e. on $[0,1]$ with the exception of finitely many points. Since $w(p)=G(0)=0$, we conclude that $w(p) = G(p)$ for any $p\in[0,1]$. This implies
\[
 W(X,\hat{Z}) = w(1) = G(1) = \sup_{p \in [0,1]} l(p) =\dist(X,A_2(Y)).
 \]
The equivalent formula for $\dist(X,A_2(Y))$ follows from Theorem~\ref{t:lower_bound}.
\end{proof}

\section{Generalization}
\label{s:generalization}

The conclusions in section~\ref{s:projection} are obtained under the assumption that $\sign(h(p))$ has finitely many discontinuity points. In this section, we show that we can remove this assumption. 
Let the random variables $X, Y\in\Lc_1(\Omega,\Fc,P)$ be fixed. 
\begin{theorem}
For the random variable $X$, the optimal value of problem \eqref{p:projection2} with $\ell=1$ and $r=2$, the distance satisfies the following formulae:
\begin{equation}
\label{e:distance-final}
\begin{aligned}
\dist (X,A_2(Y)) & = \max\Big(0, \sup_{p \in (0,1]} \big(F^{(-2)}_Y(p) - F^{(-2)}_X(p)\big)\Big) \\ 
& = \max\Big(0, \sup_{\eta \in \Rb} \big(F^{(2)}_X(\eta) - F^{(2)}_Y(\eta)\big)\Big).
\end{aligned}
\end{equation}
\end{theorem}
\begin{proof}
If $X\in A_2(Y)$, then 
\begin{gather*}
\sup_{p \in (0,1]} \big(F^{(-2)}_Y(p) - F^{(-2)}_X(p)\big)\leq 0 \quad\text{and}\quad
\sup_{\eta \in \Rb} \big(F^{(2)}_X(\eta) - F^{(2)}_Y(\eta)\big) \leq 0.
\end{gather*}
Hence, formula \eqref{e:distance-final} holds in that case as the distance $\dist (X,A_2(Y))=0$. 

We shall consider the case, when $X\not\in A_2(Y)$. 
The idea of the proof is to approximate problem \eqref{p:projection2} by a sequence of problems for which Assumption 1 is satisfied. 
We have assumed that the probability space is atomless, which entails that without loss of generality, we identify $X = F^{(-1)}_X(U)$, where $U$ has the uniform distribution on $[0,1]$. 
Define 
\[
\Omega_{i} = \{ \omega: \frac{i}{N} < U(\omega) \leq \frac{i+1}{N} \},\quad i=0, \cdots N-1
\] 
and denote the $\sigma-$algebra generated by this partition as $\mathcal{B}$. We define $X_{\varepsilon} = \mathbb{E}[X|\mathcal{B}]$.
Due to the integrability of $X$, for every $\varepsilon > 0$, we can choose $N$ such that $\|X_{\varepsilon} - X\| < \varepsilon$, where
$\|X_\varepsilon-X\|$ stands for the $\Lc_1$ norm. 
Similarly, we define the random variable $\tilde{Y} = F^{(-1)}_Y(U)$, which has the same distribution as $Y$. We partition $\Omega$ further into $MN$ subsets such that 
\[
\Omega_{ij} = \{ \omega: \frac{i}{N} + \frac{j}{MN} < U(\omega) \leq  \frac{i}{N} + \frac{j+1}{MN} \},\quad 
i=0, \cdots N-1,\; j=0,\dots, MN. 
\]
Denoting the $\sigma-$algebra generated by this finer partition by $\mathcal{\tilde{B}}$, we define $Y_{\varepsilon} = \mathbb{E}[X|\mathcal{\tilde{B}}]$.
Choosing $M$ large enough, we ensure that $\|Y_{\varepsilon} - Y\| < \varepsilon.$
Now, we look at the optimization problem:
\begin{equation}
\label{p:approximate}
\min_{Z} \;  \mathbb{E}[Z] - \mathbb{E}[X_\varepsilon] \quad 
\text{s.t. }\;  Z \succeq_{(1)} X_\varepsilon, \quad
 Z \succeq_{(2)} Y_\varepsilon. 
\end{equation}
Assumption 1 is satisfied for this problem. 
Hence, we can use the constructed projection of $X_\varepsilon$ onto $A(Y_\varepsilon)$; we denote it by ${Z}_\varepsilon$. Using Theorem~\ref{t:equivalence}, we infer that 
\[
\dist\big(X_\varepsilon,A(Y_\varepsilon)\big)=W(X_\varepsilon,Z_\varepsilon) = \sup_{\eta\in\Rb}\big(F_{X_\varepsilon}^{(2)}(\eta)- F_{Y_\varepsilon}^{(2)}(\eta)\big)= \Eb(Z_\varepsilon) -\Eb(X_\varepsilon).
\]
We have 
\[
\big|F^{(2)}_{X_\varepsilon}(\eta) - F^{(2)}_X(\eta)\big| = \big|\Eb \big[(\eta- X_\varepsilon)_+ - (\eta- X)_+\big]\big|\leq \Eb \big|(\eta- X_\varepsilon)_+ - (\eta- X)_+\big|\leq  \|X_\varepsilon - X\|.
\]
We infer that 
\[
\sup_\eta\big|F^{(2)}_{X_\varepsilon}(\eta) - F^{(2)}_X(\eta)\big| \leq \|X_\varepsilon - X\|\xrightarrow[\varepsilon\to\infty]{} 0.
\]
Analogously, $\sup_\eta\big|F^{(2)}_{Y_\varepsilon}(\eta) - F^{(2)}_Y(\eta)\big| \leq \|Y_\varepsilon - Y\|\xrightarrow[\varepsilon\to\infty]{} 0.$
Therefore, we derive the following:
\begin{align*}
\sup_{\eta\in\Rb}& \big(F_{X_\varepsilon}^{(2)}(\eta)- F_{Y_\varepsilon}^{(2)}(\eta)\big)+ \Eb(X_\varepsilon) \\ 
& = 
\sup_{\eta\in\Rb}\big(F_{X_\varepsilon}^{(2)}(\eta) + F_{X}^{(2)}(\eta) - F_{X}^{(2)}(\eta) 
+F_{Y}^{(2)}(\eta) - F_{Y}^{(2)}(\eta) - F_{Y_\varepsilon}^{(2)}(\eta)\big)+ \Eb(X_\varepsilon) \\
& \leq \sup_{\eta\in\Rb}\big(F_{X_\varepsilon}^{(2)}(\eta) - F_{X}^{(2)}(\eta)\big) + \sup_{\eta\in\Rb}\big(F_{X}^{(2)}(\eta)- F_{Y}^{(2)}(\eta)\big)\\ 
&  +\sup_{\eta\in\Rb}\big(F_{Y}^{(2)}(\eta) - F_{Y_\varepsilon}^{(2)}(\eta)\big)+ \Eb(X_\varepsilon)
 \xrightarrow[\varepsilon\to\infty]{} 
 \sup_{\eta\in\Rb}\big(F_{X}^{(2)}(\eta)- F_{Y}^{(2)}(\eta)\big) + \Eb(X)
\end{align*}
Hence, $\mathbb{E}[Z_\varepsilon]$ has a limit when $\varepsilon\to 0$. 

Due to the construction of the variables $Y_\varepsilon$, the following relations hold: 
\[
{Z}_\varepsilon  \succeq_{(2)} Y_\varepsilon \succeq_{(2)} Y.
\]
The second order dominance $Y_\varepsilon \succeq_{(2)} Y$ follows by virtue of Jensen's inequality: 
\[
\mathbb{E}[(\eta - \mathbb{E}[Y|\mathcal{\tilde{B}}])_+] \leq \mathbb{E}[(\eta - Y)_+].
\]
According to Lemma~\ref{l:smallerdist}, variables $\hat{Z}_\varepsilon$ exist such that $\hat{Z}_\varepsilon$ are feasible for 
problem \eqref{p:approximate} and $W(X,\hat{Z}_\varepsilon)\leq W(X,Z_\varepsilon).$
Since $Z_\varepsilon\in A_2(Y)$, we have
\begin{align*}
\dist(X,A_2(Y))& \leq W(X,\hat{Z}_\varepsilon)\leq W(X,Z_\varepsilon) \leq W(X_\varepsilon,{Z}_\varepsilon) + W(X_\varepsilon,X)\\ 
& = \sup_{\eta\in\Rb}\big(F_{X_\varepsilon}^{(2)}(\eta)- F_{Y_\varepsilon}^{(2)}(\eta)\big) + W(X_\varepsilon,X).
\end{align*}
Using Theorem~\ref{t:lower_bound}, we obtain
\[
\sup_{\eta\in\Rb}\big(F_{X}^{(2)}(\eta)- F_{Y}^{(2)}(\eta)\big) \leq \dist(X,A_2(Y))\leq \sup_{\eta\in\Rb}\big(F_{X_\varepsilon}^{(2)}(\eta)- F_{Y_\varepsilon}^{(2)}(\eta)\big) + W(X_\varepsilon,X).
\]
Passing to the limit with $\varepsilon\to\infty$ and having in mind that convergence in $\Lc_1$ implies convergence with respect to $W_1$, we obtain the result.
\end{proof}

\section{Numerical Solution of the Transportation-Distance Relaxations}
\label{s:method}

In this section, we propose a method for the numerical solution of problem \eqref{p:relaxedproblem} and show its convergence. 

Given the analytic expression for the distance $\dist(X,A_2(Y))$, we reformulate problem \eqref{p:relaxedproblem} as follows.
\begin{equation}
\label{p:relaxedp2}
\min_{x} \, f(x) + \alpha \dist(G(x),A_2(Y)) \quad \text{s.t. }  x \in \mathcal{X}. 
\end{equation}
We assume that the objective function $f : \Rb^n \to \Rb$ is a convex function, $\Xc \subseteq \Rb^n$ is a compact convex set. The random variable $Y \in \Lc_1(\Omega, \Fc, P)$ is a given benchmark. The operator $G(\cdot)$ is continuous with respect to the $\Lc_1-$norm and it has the form
\[
[G(x)](\omega) = g(x,\omega),
\]
where $g:\Rb^n\times \Omega\to\Rb$ is such that $g(x,\cdot)$ is $\Fc$-measurable for all $x\in\Rb^n$ and $g(\cdot,\omega)$ is a concave function for all $\omega\in\Omega.$ As $g(\cdot,\omega)$ is a concave function and the set $\Xc$ is compact, $g(\cdot,\omega)$  is globally Lipschitz continuous on $\Xc$ with an associated constant $L(\omega)>0.$ We assume that $L(\omega)$ is an integrable random variable.
Furthermore, since $g(\cdot, \omega)$ is a finite-valued concave function, it is continuous everywhere.

Recall that 
$\dist(G(x),A_2(Y)) = \max\Big(0,\sup_\eta \big\{\Eb[(\eta-G(x))_+] - \Eb[(\eta-Y)_+]\big\}\Big).$ 

For a fixed $\omega\in\Omega$, the subdifferential $\partial g(\bar{x},\omega)$ is given by
\[
\partial g(\bar{x},\omega) = \{ s\in\Rb^n: g(x,\omega) \leq g(\bar{x},\omega) + \langle s(\omega), x-\bar{x}\rangle\; \forall x\in\Rb^n\}.
\]
The subdifferential $\partial G(\bar{x})$ is then given by
\[
\partial(G(\bar{x})) = \big\{ S:\Rb^n\to \Lc_1(\Omega,\Fc,P): [Sd](\omega) = \langle s(\omega),d\rangle\; \text{for some } s(\omega)\in\partial g(\bar{x},\omega)\big\}.
\]
Let the multifunction $DH(Z,\eta,\omega)$ for a random variable $Z$ be defined as 
\[
DH(Z,\eta,\omega) =
\begin{cases}
\{-1\} &  \text{if } Z(\omega) < \eta,\\
[-1,0] & \text{if } Z(\omega) = \eta,\\
\{ 0\} & \text{if } Z(\omega) > \eta.
\end{cases}
\]
The subdifferential of $\Eb[\max(0,\eta - G(\bar{x}))]$ takes on the form 
\begin{equation}
\label{e:subdif_eshortfall}
\begin{aligned}
\partial \Eb[\max(0,\eta - G(\bar{x}))] = \big\{ -\Eb[S\xi]:~ & S\in\partial G(\bar{x}),\; \xi \in \Lc_\infty(\Omega,\Fc,P) \\ &\text{with } \xi(\bar{x},\eta,\omega)\in DH(G(\bar{x}),\eta,\omega) \big\}.
\end{aligned}
\end{equation}
To simplify notation, we introduce the following objects:
\begin{align*}
v(\eta)& = \Eb[\max(0,\eta - Y)]\quad \text{for } \eta\in\Rb;\\
\D(x) & = \sup_\eta\big[\Eb[(\eta-G(x))_+] - v(\eta)]\big];\\
J(x) &= \Big\{\eta\in \Rb: \D(x) = \Eb[(\eta-G(x))_+] - v(\eta)>0\Big\}.
\end{align*}
The convex subdifferential of $\D(x)$ has the following form:
\[
 \partial \D (x) = \conv \big\{\cup_{\eta\in J(x)} -\Eb[S\xi]\big\},
 \]
 where $S$ and $\xi$ are given in \eqref{e:subdif_eshortfall}. 
We adopt the convention that the subdifferential of $\D(x)$ contains a zero vector, if $J(x)=\emptyset$.\\

We propose the following method with a parameter $\alpha>0$ for the numerical solution of problem \eqref{p:relaxedp2}.
\begin{description}
\item[]\vspace{0.3cm} \hrule \vspace{0.17cm}
            \textbf{Transportation Distance Relaxation Method}
            \vspace{0.17cm} \hrule \vspace{0.2cm}
\item[\emph{Step 0.}] Set $k=1$ and  $x^0\in\Xc$. Calculate $f(x^0),\,\D(x^0)$
$s_f^0\in\partial f(x^0)$, $s_p^0\in \partial \D(x^0)$ and set $J_f^1=J_p^1=\{ 0\}$.
\item[\emph{Step 1.}]
Solve problem \eqref{master-penalty} and let $({x}^k, \theta_f^k,\theta_p^k)$ denote its solution.
\begin{equation}
\label{master-penalty}
\begin{aligned}
\min_{x,\theta_f,\theta_p}\; &  \theta_f + \alpha \theta_p   \\
\text{s.t.} \ & f(x^j)+ \langle s_f^j, x-x^j\rangle \leq \theta_f,\quad j\in J_f^k,\\
& \D(x^j)+ \langle s_p^j, x-x^j\rangle \leq \theta_p,\quad j\in J_p^k.\\
& x \in \Xc, \;\; \theta_p\geq 0,\;\; \theta_f\in\Rb.
\end{aligned}
\end{equation}
\item[\emph{Step 2.}] For $x^k$, calculate $\D(x^k)$ and $f(x^k)$. 
If $f(x^k)=\theta_f^k$ and either $\D(x^k)\leq 0$ or $\D(x^k) = \theta_p^k$, then stop; otherwise continue.
\item[\emph{Step 3.}]
If $\D(x^k) >\theta_p^k$ then calculate $s_p^{k}\in \partial \D(x^k)$, and set $J_p^{k+1} = J_p^{k}\cup\{k\}$, otherwise set $J_p^{k+1} = J_p^{k}$. 
If $f(x^k) > \theta_f^k,$ then calculate $s_f^k\in\partial f(x^k)$ and set $J_f^{k+1} = J_f^{k-1}\cup\{k\}$, otherwise set $J_f^{k+1} = J_f^{k}$. 
\item[\emph{Step 4.}]
Increase $k$ by one and return to Step 1.
\vspace{0.3cm} \hrule \vspace{0.3cm}
\end{description}

\begin{theorem}
\label{t:convergencePFAM}
Assume that the set $\Xc$ is nonempty, convex and compact, the function $f(\cdot)$ is convex and $G(\cdot)$ is a continuous and concave operator on $\Xc$. Then the Transportation Distance Relaxation Method either stops with an optimal solution of problem \eqref{p:relaxedp2} or it generates a sequence of points $\{x^k\}$, which converges to an optimal solution of that problem.
Additionally, if problem 
\begin{equation}
\label{p:mainp}
\min f(x) \quad \text{s.t. }\; G(x)\ssd Y,\;\; x\in\Xc, 
\end{equation}
is feasible, then for any parameter $\alpha \geq L_f$, the method identifies an optimal solution of problem \eqref{p:mainp} with $L_f$ being the Lipschitz constant of $f$ over the set $\Xc$. 
\end{theorem}
\begin{proof}
If the method stops at Step 2 in iteration $k$ with $\D(x^k)\leq 0,$ then $G(x^k)\in A_2(Y)$, implying that problem \eqref{p:mainp} is feasible. Under the assumption of the theorem, the set $\Xc\cap \{ x\in\Xc: G(x)\in A_2(Y)\}$ is closed. Indeed, let $\{x^m\}_{m\in\Nb}\in\Xc$ be a sequence converging to $\bar{x}$. Since $\Xc$ is compact, $\bar{x}\in\Xc$. Due to the continuity of $G(\cdot)$, we infer that $G(x^m)$ converges to $G(\bar{x})$ in $\Lc_1$. The set $ A_2(Y)$ is closed with respect to $\Lc_1$ -convergence entailing that $G(\bar{x})\in A_2(Y)$. Therefore, problem \eqref{p:mainp} has a non-empty and compact  feasible set. This implies that  problem \eqref{p:mainp}  has an optimal solution; denote it by $x^*.$

Since $x^k$ is feasible according to the stopping case, we infer that $f(x^k)\geq f(x^*)$. 
Furthermore, the approximation of the function $f(\cdot)$ given by $\theta_f^k$ is exact at $x^k$. For all points $x$ that are feasible for \eqref{p:relaxedp2}, we have
\[
f(x)\geq \max_{j\in J_f^k} \big(f(x^j)+ \langle s_f^j, x-x^j\rangle\big) = \theta_f^k
\]
implying that $f(x^*)\geq \theta_f^k.$ Hence, $f(x^k)\geq f(x^*)\geq \theta_f^k = f(x^k)$, entailing that $x^k$ is an optimal solution of problem \eqref{p:mainp}.\\

Under the assumption of the theorem problem \eqref{p:relaxedp2} has an optimal solution; denote it by $\hat{x}$.  If the method stops at Step 2 at iteration $k$ and $\D(x^k)> 0$, then $\D(x^k)= \theta_p^k$ and $G(x^k)\not\in A_2(Y)$. 
Using the subgradient inequality and the optimality of $x^k$ for problem \eqref{master-penalty} at iteration $k$, we obtain the following chain of inequalities. 
\begin{align*}
f(\hat{x})+ & \alpha\dist (G(\hat{x}), A_2(Y)) \\  
& \geq \max_{j\in J_f^k} \big(f(x^j)+ \langle s_f^j, \hat{x}-x^j\rangle\big) + \alpha \max_{j\in J_p^k} \big(\D(x^j)+ \langle s_p^j, \hat{x}-x^j\rangle\big)\\ 
& \geq \max_{j\in J_f^k} \big(f(x^j)+ \langle s_f^j, x^k-x^j\rangle\big) + \alpha \max_{j\in J_p^k} \big(\D(x^j)+ \langle s_p^j, x^k-x^j\rangle\big)\\ 
& = \theta_f^k+\alpha\theta_p^k.
\end{align*}
Hence, if the method stops at Step 2 with $\theta_p^k >0$, then it delivers an optimal solution of problem \eqref{p:relaxedp2}. 

The function $f(\cdot)$ and the mapping $G(\cdot)$ are Lipschitz continuous over the set $\Xc$. If the set $\Xc\cap \{ x\in\Xc: G(x)\in A_2(Y)\}$ is non-empty  and  $\alpha \geq L_f$, where $L_f$ is the Lipschitz constant of the function $f$ over the set $\Xc$, then 
the Clarke's exact penalty function theorem (\cite[Proposition 2.4.3]{clarkeoptimization}) is applicable. 
The theorem implies that problem \eqref{p:mainp} has an optimal solution $x^*$  if and only if it solves also problem \eqref{p:relaxedp2}. Hence, if the method stops with $\D(x^k)=\theta_p^k > 0$ and $\alpha$ is large enough, than \eqref{p:mainp} is infeasible. 

Now, assume that the method generates an infinite sequence of points. 
Two cases are possible: a number $k_0$ exists such that the set $J(x^k)=\emptyset$ for all $k\geq k_0$, or such a number does not exist.
In the first case, the set $J_p^k$ remains the same for all $k\geq k_0$. We observe that $\theta_p^k$ remain constant and all points $x^k$ for $k\geq k_0$ are feasible for problem \eqref{p:relaxedp2}: $G(x^k)\in A_2(Y)$ holds for all $k\geq k_0$. Then the method converges to an optimal solution of problem \eqref{p:mainp}, where the convergence of the method follows from the convergence properties of the cutting plane method applied to the function $f$ and the closedness of the feasible set of  problem \eqref{p:mainp}. In the case, when the set $J_p^k$ grows infinitely the convergence statement follows from the convergence of the cutting plane method applied to problem \eqref{p:relaxedp2}. In that case, again an accumulation point $\hat{x}$ of the sequence $\{x^k\}$ exists and the corresponding values $\theta_f^k$ and $\theta_p^k$ converge to $f(\hat{x})$ and $\D(\hat{x})$, respectively. The same observations of the infeasibility of  problem \eqref{p:mainp} hold as in the case of stopping with a positive value for $\theta_p^k.$ 
\end{proof}

Some remarks are in order.
Notice that the calculation of $\D(x)$ might not be trivial as we need to maximize over the entire real line.
However, due to the special structure of the function, $\D(x)$ is always bounded: 
$|\D(x)| \leq |\Eb[G(x)]-\Eb[Y]|$. As both shortfall functions $\Eb[(\eta-G(x))_+]$ and $v(\eta)$ converge to zero when $\eta\to -\infty$ and have an asymptote when $\eta\to\infty$, we could choose a very large interval $[a,b]$ and maximize the difference over $[a,b]$. The task becomes much simpler if the probability space is finite. Then the shortfall functions are piecewise linear and we only need to check the difference $\Eb[(\eta-G(x))_+]-v(\eta)$ at their break points in order to determine the maximum. 

We assume that the decision maker chooses a weight parameter $\alpha$ that represents a reasonable compromise between minimizing the objective function and achieving an outcome with a distribution close to the desired one. This compromise would be a modelling choice depending on context and would provide a specific value of the parameter $\alpha$.  
If the infeasibility of problem \eqref{p:mainp} is not known and we desire to achieve feasibility at all cost if possible at all, then we could modify the proposed method to increase the parameter $\alpha$
if $\D(x^k) >0$. In this way, the weight of the distance in the objective function would increase. If problem \eqref{p:mainp} is feasible the weight of the penalty will eventually become large enough and the modified method will stop at step 2 with the optimal solution of it.  If the method does not stop, then it will produce a sequence $x^k$ with an accumulation point $\bar{x}\in\Xc$. Let $x^m$ with $m\in\Kc\subset\Nb$  be the subsequence converging to $\bar{x}$ with $\Nb$ being the set of natural numbers. Due to the properties of the cutting plane method, the method will converge meaning that $\lim_{m\in\Kc} f(x^m) = f(\bar{x})=\bar{\theta}_f$  and  $\lim_{m\in\Kc} \D(x^m) = \bar{\theta}_p = \min_{x\in\Xc} \D(x) = \D(\bar{x})$. 
If one wishes to increase the penalty parameter $\alpha$ in the course of optimization, one could  choose $\alpha_{k+1} = \max\big(\gamma\alpha_k,\max_{j\in J_f^k}\|s_f^k\|\big)$ for some $\gamma>1$.

\section{Numerical experiments}
\label{s:numerical}

We implemented and tested our method on several test problems focusing mainly on the two-stage problem with a stochastic dominance constraint on the recourse function. The method is implemented in Python with Gurobi as the optimization solver for the master problem \eqref{master-penalty} in Step 1.

In the first experiment, we consider an inspection design which includes 
performing $N$ possible searches. Each search can take time $x_n$ and results in a unit cost of $c_n$ for the  equipment and experts labor involved. The total weighted search time $w^\top x$ is bounded by $M$. The searches increase the quality of the inspection, which is a random quantity depending on $s$ scenarios and on the time vector $x$. The dependence is given by the mapping $Q(x)$ that has concave realizations: functions $q_s(x)$, where $s$  stands for the respective scenario, $s= 1, \cdots, S$. We aim to minimize the expected search time and cost while ensuring good quality. We use a benchmark $Y$ for the resulting distribution.  
The problem formulation, which is a counterpart of \eqref{p:mainp}, is the following:
\begin{align*}
\min_{x} \, &  c^\top  x - \kappa\Eb[Q(x)] \\
\text{s.t. } & G(x)  \ssd Y,\\
&  w^\top x \leq M, \;\; x \geq 0.
\end{align*}
Here $\kappa>0$ is a weighing parameter and $G(x) = Q(x)-c^\top x.$ 
The optimization problem with a relaxation of the dominance constraint is formulated as follows: 
\begin{align*}
\min_{x} \, &  c^\top  x - \kappa\Eb[G(x)] + \alpha \D (x) \\
\text{s.t. }  
& \D(x) \geq   \sum^{S}_{s=1} p_s \max(0, y_j - g_s(x)) - v(y_j)\quad j=1,\dots, S. \\
& w^\top x \leq M, \;\; x \geq 0.
\end{align*}
Recall that $v(y_j) =\Eb[\max(0,\eta - y_j)]$ with $y_j$ being the $j$-th realization of the benchmark $Y$. 
The method is tested with $N=20$, $\kappa=\alpha=1$ and $S=20$, as well as $S=100$. The data is randomly generated. The results are reported in Table~\ref{tab:inspection}
The cumulative distribution function and the shortfall functions $F^{(2)}$ of the benchmark, the solution, and the projection onto the set $A_2(Y)$ are shown in Figures \ref{fig:F1F2_S100} and \ref{fig:F1F2_S300}. 
\begin{table}[h]
\centering
\begin{tabular}{ccccccc}
\hline
 \#Scen. & Expected & Expectation& Expectation  & Distance & CPU & \#Iter.\\
           & outcome  & Projection & benchmark & (sec.)   & \\
\hline
$100$ & $0.2393$ & $0.2533 $ & $0.2532$ & 0.0139 & 0.54 & 8\\
$300$ & $0.0226$ & $0.0270 $ & $0.0270$ & 0.0125 & 3.55 & 8\\
\hline
\end{tabular}
\caption{Experiments data of the inspection design problem}
\label{tab:inspection}
\end{table} 
\begin{figure}[h!]
\centering
\includegraphics[scale=0.41]{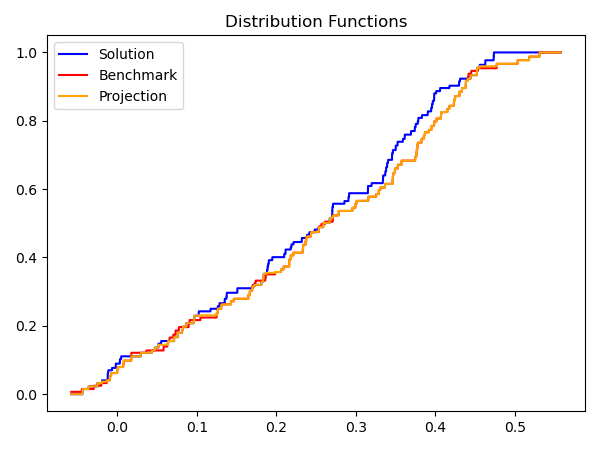}
  \hspace{0.3ex}
  \includegraphics[scale=0.41]{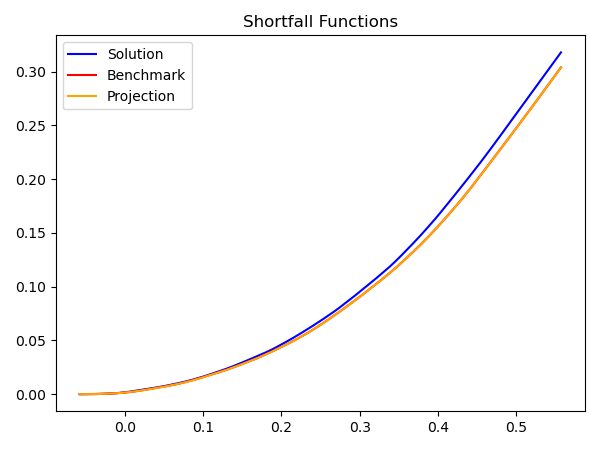}
  \caption{Cumulative distribution functions and the shortfall functions for the case of $S=100$ }
  \label{fig:F1F2_S100}
\end{figure}
\begin{figure}[h!]
\centering
  \includegraphics[scale=0.41]{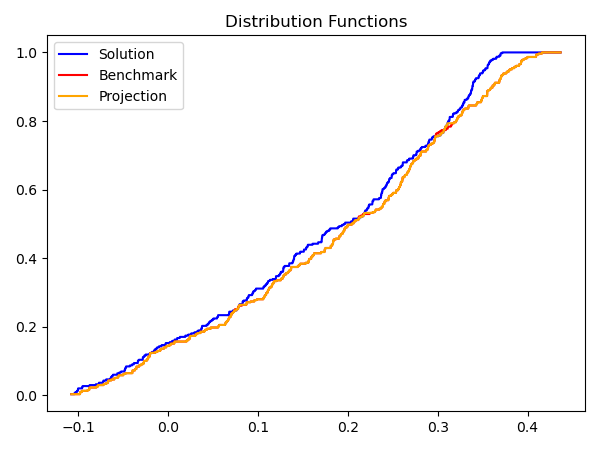}
  \hspace{0.3ex}
  \includegraphics[scale=0.41]{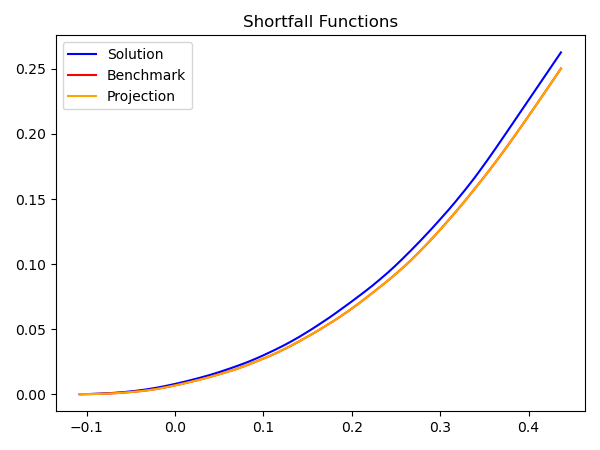}
  \caption{Cumulative distribution functions and the shortfall functions for the case of $S=300.$ }
  \label{fig:F1F2_S300}
\end{figure}

Our second experiment addresses an emergency relief problem, in which a number $I$ of centers deploy units that provide relief if disaster happens in $T$ target areas. Each center may move emergency supply to each target, but within their total capacities and at different costs.
At each target, the demand for emergency supply units is random. In addition to that, the deployment routes may be destroyed and parts of the supply units moved may become unavailable. However, the decisions how to deploy supply units and to choose their routes must be made before the demand or the state of the routes are known. 
The center capacities are denoted by $C_i$ and the unit deployment costs is denoted by $c_i$, $i=1,\dots I$. The unit moving costs from center $i$ to target are $j$ is denoted $q_{ij}$, $i=1,\dots I$, $j=1,\dots T$.  Furthermore, every transportation link may fail with probability 0.3, independently of other links and the demands. If the connection is intact, its throughput is 1, if the connection failed,
its throughput is 0.5, i.e., in the case of failure, the shipment along this link loses half of its value.
We index the route scenarios by $r\in \{1,\dots,N\}$ and denote the route throughput in scenario $r$ by $a_{ij}^r$.
The demand at each target is independent of the other areas and uniformly distributed with finitely many realization $d_j^s$, $j=1,\dots T$ and $s$ being a realization of the joint demand, $s=1,\dots S$.  
At each target, if the demand for relief exceeds the available shipments, the cost of $L$ per each missing unit is incurred.
We denote the numbers of units deployed at center $i$ by $x_i$, $i=1,\dots, I$, and the number of units sent from center $i$ to target $j$, $j=1,\dots,T$ by $v_{ij}$.

First, we formulate a linear programming problem to minimize the expected cost of the relief operation. Let $\sigma_j^{sr}$ represent the shortfall occurring at target $j$ 
under the demand scenario $s$ and the route scenario $r$. We introduce the shorthand notation 
\[
Q^{sr} (x,v,\sigma) = \sum_{i=1}^I c_i x_i + \sum_{i=1}^I\sum_{j=1}^T q_{ij} v_{ij} + L\sum_{j=1}^T \sigma_j^{sr} 
\] 
for the realization of the total cost in each scenario $(s,r)$. 
Then, the risk-neutral problem takes on the following form:
\begin{equation}
\label{p:exp_value_er}
\begin{aligned}
\min_{x,v,\sigma}\; &  \frac{1}{SN} \sum_{s=1}^S\sum_{r=1}^N  Q^{sr} (x,v,\sigma) \\
\text{s.t.}\ & \sum_{j=1}^T  v_{ij} \le x_i, \quad i = 1,\dots,I,\\
             & \sum_{i=1}^I a_{ij}^{r} v_{ij} + \sigma_j^{sr} \ge d_j^{s}, \quad j=1,\dots,T,\; s=1,\dots,S,\; r=1,\dots,N,\\
             & 0 \le x_i \le C_i, \quad i = 1,\dots,I, \quad v \ge 0,\quad \sigma\ge 0.
\end{aligned}
\end{equation}
As the penalty cost is random, we would like to avoid excessive cost. 
Let $B$ denote a benchmark random variable, which has a desirable distribution for the total cost. We add a constraint $Q\preceq_{\rm ico} B$, where the increasing convex order relation $Q\preceq_{\rm ico} B$ reflects our risk aversion. 
Denoting the feasible set of problem \eqref{p:exp_value_er} by $\Xc$, we formulate the risk-averse problem as follows:
\begin{equation}
\label{p:dc_er}
\min_{x,v,\sigma}\;  \frac{1}{SN} \sum_{s=1}^S\sum_{r=1}^N  Q^{sr} (x,v,\sigma) \quad
\text{s.t. }\;  Q\preceq_{\rm ico} B,\;\;
              (x,v,\sigma)\in\Xc.
\end{equation}
Using relation \eqref{e:ico}, we convert the constraint with the increasing convex order to a second-order stochastic dominance constraint as follows $-Q\ssd - B$. 
The relaxation of problem \eqref{p:dc_er} takes on the form 
\begin{equation}
\label{p:relaxed_dc_er}
\min_{x,v,\sigma}\;  \frac{1}{SN} \sum_{s=1}^S\sum_{r=1}^N  Q^{sr} (x,v,\sigma) +  \alpha \dist(-Q, A_2(-B)) \quad
\text{s.t. }\;  (x,v,\sigma)\in\Xc.
\end{equation}
We have solved the problem for $I=2$, $J=3$, and 2700 demand-route scenarios ($N=100$ and $S=27$). 

In the first series of experiments, we have created a benchmark, which results in a non-empty feasible set. First, we solved problem ~\eqref{p:relaxed_dc_er} by the Shortfall Event Cut Method described in \cite[Sec. 7.2.1]{DDARriskbook}, which took three iterations. We solved problem ~\eqref{p:relaxed_dc_er} by the proposed Transportation Distance Relaxation Method for several values of the penalty parameter. The method identified feasibility and reached the same solution as the Shortfall Event Cut Method  for the penalty parameter $\alpha\geq 50$. We report the results for the risk-neutral solution ($\alpha =0$), for two intermediate solutions, and for the case of $\alpha=50$ in Table~\ref{tab:feasible}. The distribution functions of the optimal total cost for those experiments and their integrated survival functions $F^{(2)}$ (also called excess functions) are displayed in Figure~\ref{fig:feasible_distfunctions}.
The CPU model is Intel(R) Core(TM) i7-8569U CPU @ 2.80GHz. The reported CPU time includes intermediate output from Gurobi. 
\begin{table}[h!]
\centering
\begin{tabular}{lcccc}
\hline
 Parameter & Expected cost & Distance  & Number of Iterations & CPU (sec.)\\
\hline
$\alpha=0$ & $2837.33$ & $20.00 $ & 2 & 12.9\\
$\alpha=5$ & $2867.64$ & $6.58$ & 10 & 63.9\\
$\alpha=15$ & $2894.44$ & $1.70$ & 10 & 64.0\\
$\alpha=50$ & $2948.00$ & $0.0$ & 5 & 31.7\\
$\mathrm{benchmark}$ & $3062.56$ & $0$ & - & -\\
\hline\\
\end{tabular}
\caption{Experiments with an achievable benchmark}
\label{tab:feasible}
\end{table} \vspace{-2ex}
\begin{center}
\begin{figure}[h!]
    \includegraphics[width=0.48\textwidth]{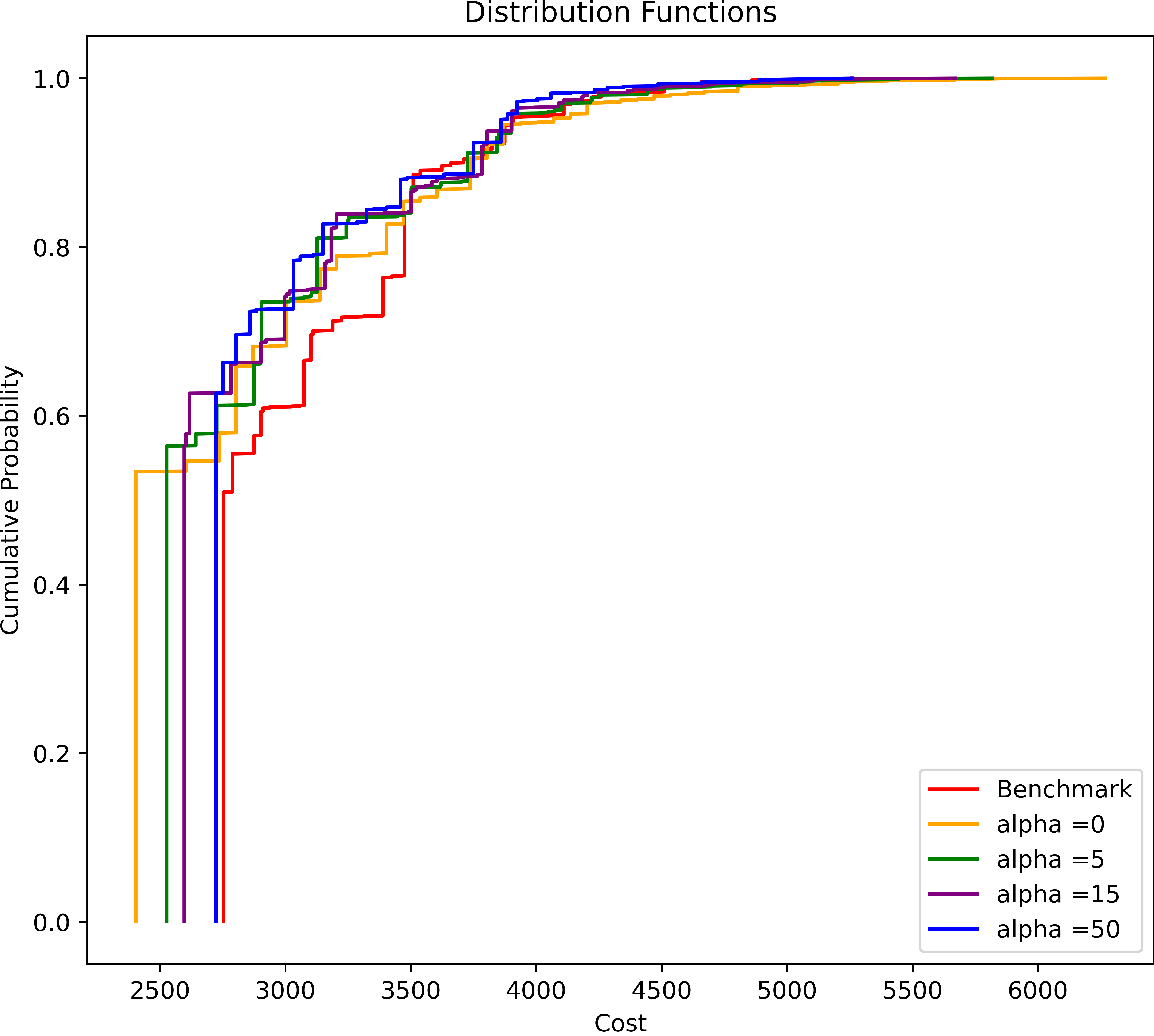} 
    \hspace{1ex}
    \includegraphics[width=0.48\textwidth]{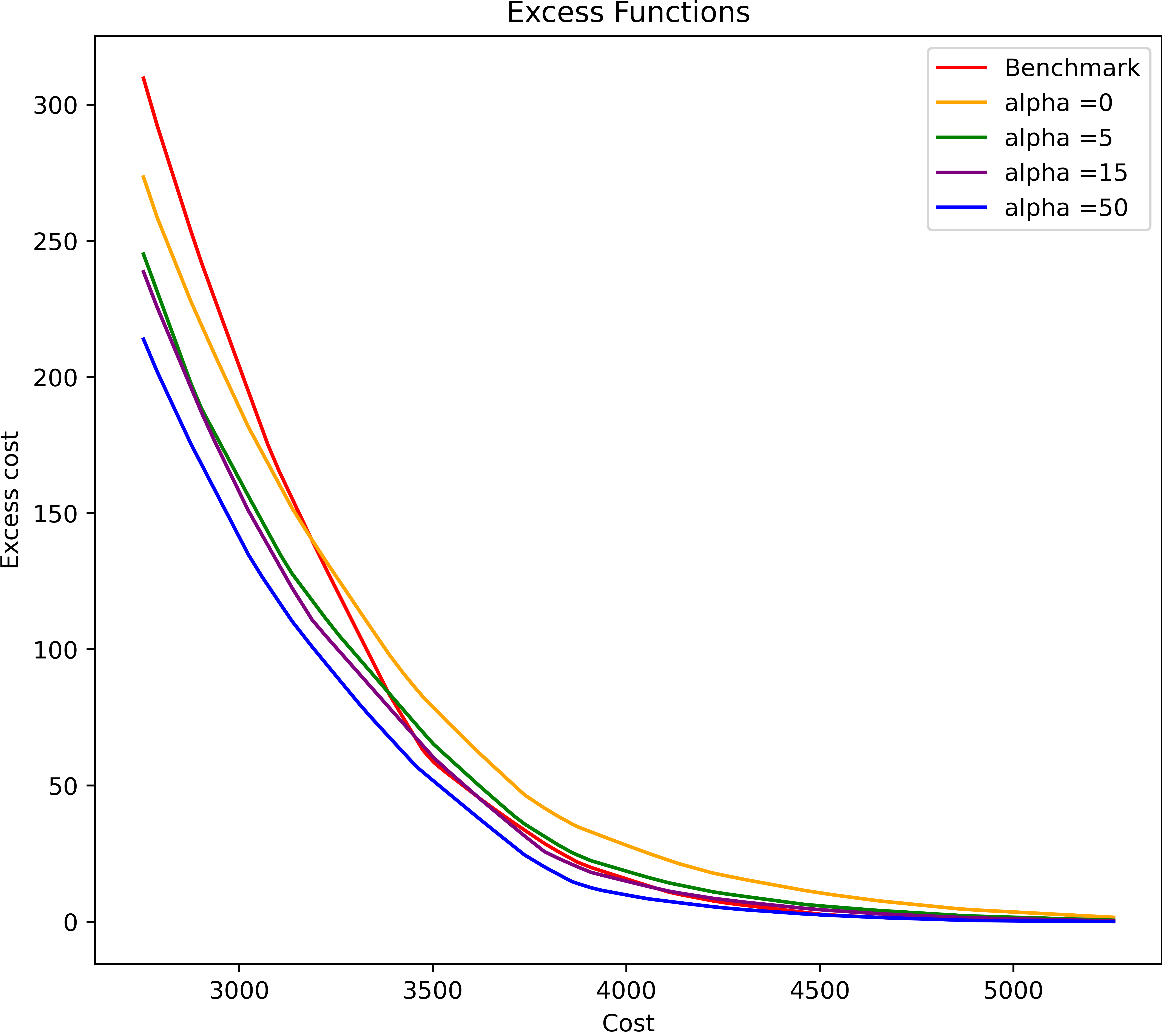} 
    \caption{Experiments with an achievable benchmark: cumulative distribution functions $F_Q$ and the excess functions $\bar{F}_Q^{(2)}$ of the optimal total cost for different values of the penalty parameter.}
    \label{fig:feasible_distfunctions}
\end{figure}
\end{center}
In the second series of experiments, we have created a benchmark, which results in an empty feasible set. We used the optimal cost of the risk neutral solution and modified the outcomes exceeding its expected value by moving them closer to the expected value.  We solved problem ~\eqref{p:relaxed_dc_er} by the proposed Transportation Distance Relaxation Method for several values of the penalty parameter.  We report the results for the risk-neutral solution ($\alpha =0$) and for the solutions for various weight parameters in Table~\ref{tab:infeasible}. The distribution functions of the optimal total cost for those experiments and the respective excess functions are displayed in Figure~\ref{fig:infeasible_distfunctions}.
{
\begin{table}[h]
\centering
\begin{tabular}{lcccc}
\hline
 Parameter & Expected cost & Distance  & Number of Iterations & CPU (sec.)\\
\hline
$\alpha=0$ & $2837.33$ & $79.69 $ & 2 & 13.2\\
$\alpha=0.5$ & $2839.64$ & $71.72$ & 3 & 20.2\\
$\alpha=1$ & $2842.67$ & $67.43$ & 1 & 6.66\\
$\alpha=5$ & $2923.03$ & $45.21$ & 9 & 60.5\\
$\mathrm{benchmark}$ & $2877.82$ & $0$ & - & -\\
\hline\\
\end{tabular}
\caption{Experiments with benchmark resulting in infeasibility}
\label{tab:infeasible}
\end{table}
\begin{center}
\begin{figure}
    \includegraphics[width=0.48\textwidth]{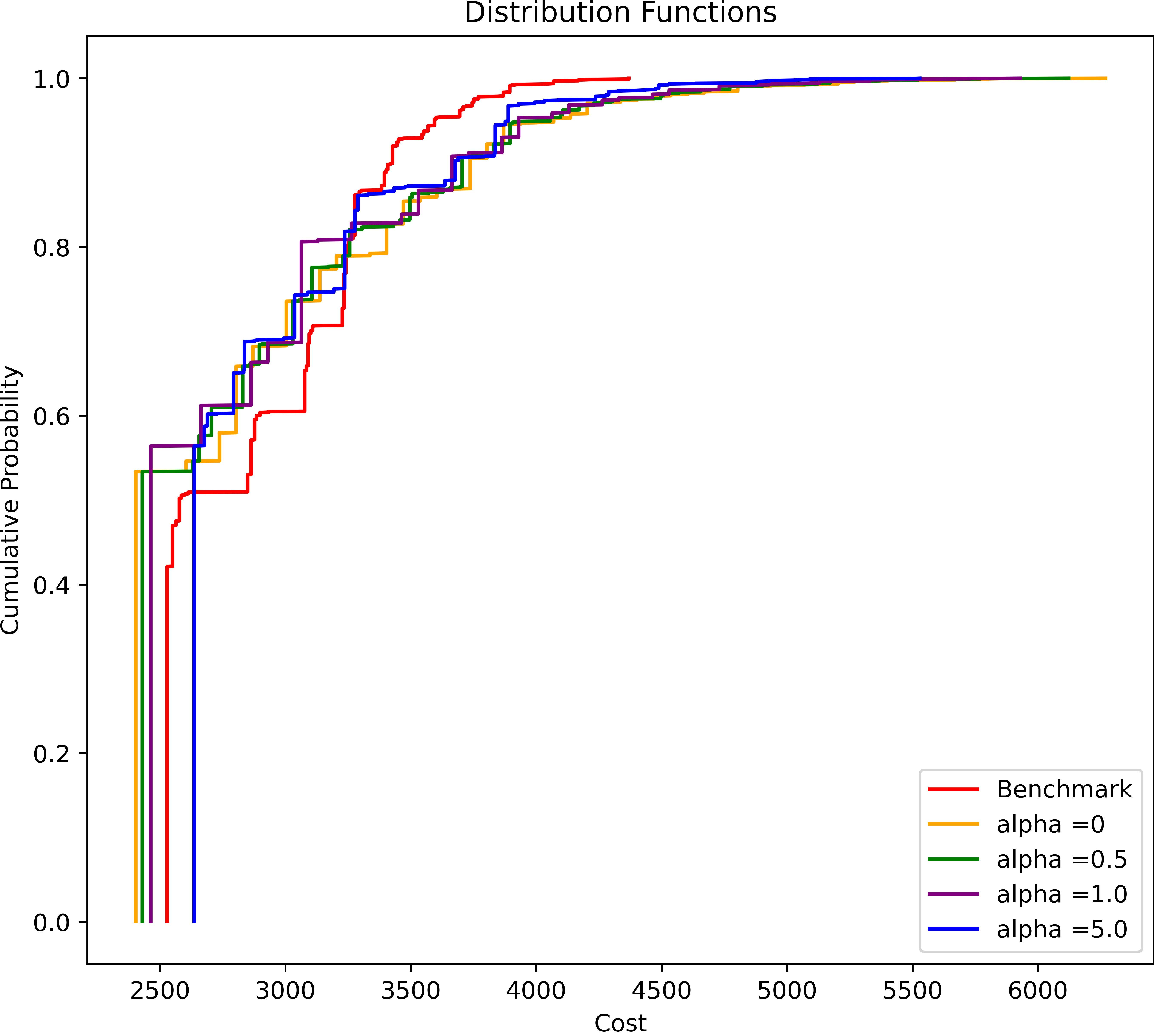} 
    \hspace{1ex}
    \includegraphics[width=0.48\textwidth]{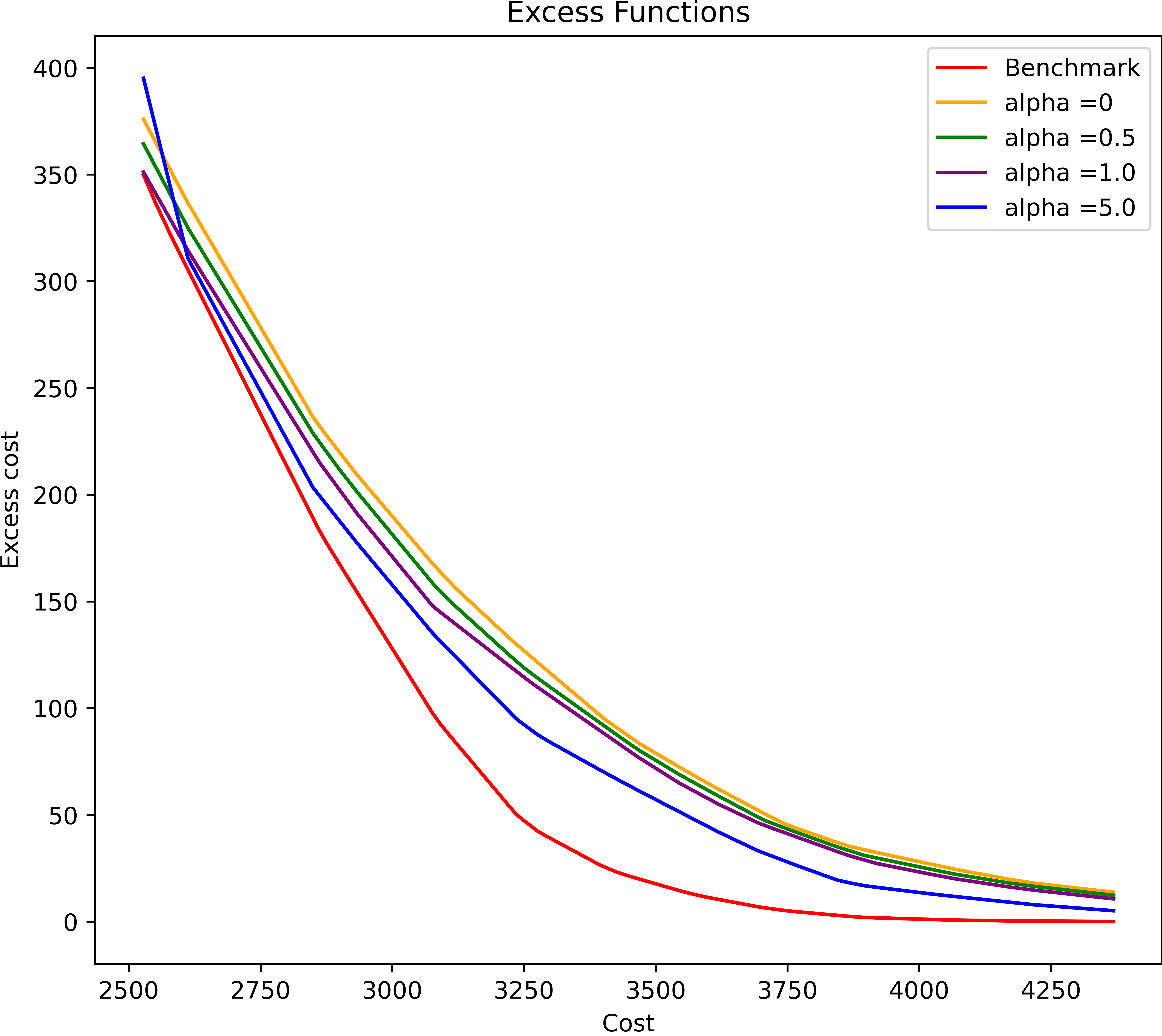} 
    \caption{Experiments with infeasible benchmark: cumulative distribution functions $F_Q$ and the excess function $\bar{F}_Q^{(2)}$ of the optimal total cost for different values of the penalty parameter. }
    \label{fig:infeasible_distfunctions}
\end{figure}
\end{center}
}
We also solved the problem with the penalty parameter $\alpha > 50$ and obtained solutions identical to the solution obtained with $\alpha=50.$ 

We conclude that the Transportation Distance Relaxation Method is a good way to relax a benchmark when the stochastic dominance constraint leads to infeasibility of the optimization problem. We also note that the computational effort of our method for the relaxed problem is comparable to the computational effort for solving problems with stochastic dominance constraint presented in \cite[Chapter 7]{DDARriskbook}.

\section*{Acknowledgments}
The authors thank the Associate Editor and the three anonymous referees whose comments helped improve the paper. 
The support of the Air Force Office of Scientific Research under award FA9550-24-1-0284 is gratefully acknowledged.

\bibliographystyle{plain}

\end{document}